\documentclass[11pt]{amsart}

\usepackage{epsfig,amssymb}

\makeatletter
\def\eqalign#1{\null\vcenter{\def\\{\cr}\openup\jot\m@th
  \ialign{\strut$\displaystyle{##}$\hfil&$\displaystyle{{}##}$\hfil
      \crcr#1\crcr}}\,}
\makeatother

\newcommand{\be}{\begin{equation}}
\newcommand{\ee}{\end{equation}}
\newcommand{\beq}{\begin{eqnarray}}
\newcommand{\eeq}{\end{eqnarray}}
\newcommand{\bt}{\beta}
\newcommand{\bl}{\begin{lemma}}
\newcommand{\el}{\end{lemma}}
\newcommand{\bm}{\begin{pmatrix}}
\renewcommand{\em}{\end{pmatrix}}
\newcommand{\bml}{\begin{multline}}
\newcommand{\eml}{\end{multline}}
\newcommand{\ba}{\begin{array}}
\newcommand{\ea}{\end{array}}

\newcommand{\la}{\label}
\newcommand{\ci}{\cite}

\newcommand{\de}{\delta}
\newcommand{\De}{\Delta}
\newcommand{\al}{\alpha}
\newcommand{\ga}{\gamma}
\newcommand{\Ga}{\Gamma}

\newcommand{\si}{\sigma}
\newcommand{\Si}{\Sigma}
\newcommand{\om}{\omega}
\newcommand{\Om}{\Omega}
\newcommand{\lb}{\lambda}
\newcommand{\ze}{\zeta}

\renewcommand{\th}{\theta}

\newcommand{\ep}{\varepsilon }

\newcommand{\bi}{\bibitem}

\newfont{\msbm}{msbm10 scaled\magstep1}%blackboardbold
\newfont{\msbms}{msbm7 scaled\magstep1} %blackboardbold   subscript

\newcommand{\bbr}{\mbox{$\mbox{\msbm R}$}}

\newcommand{\bbc}{\mbox{$\mbox{\msbm C}$}}

\newcommand{\bbz}{\mbox{$\mbox{\msbm Z}$}}

\newtheorem{theorem}{Theorem}[section]
\newtheorem{lemma}[theorem]{Lemma}
\newtheorem{proposition}[theorem]{Proposition}

\theoremstyle{definition}

\newtheorem{example}[theorem]{Example}

\theoremstyle{remark}
\newtheorem{remark}[theorem]{Remark}

\numberwithin{equation}{section}

\begin{document}
\def\wt{\widetilde}
%\hfill {\small February 28, 2008}
\title[Toeplitz determinants]{Asymptotics of Toeplitz, Hankel, and Toeplitz+Hankel determinants
with Fisher-Hartwig singularities}
\author{P. Deift}
\address{Courant Institute of Mathematical Sciences, New York, NY, USA}
\author{A. Its}
\address{
Indiana University -- Purdue University  Indianapolis,
Indianapolis, IN, USA}
%\thanks{The second author was supported in part by NSF Grant \#DMS-0401009.}
%
\author{I. Krasovsky}
\address{
Brunel University West London, Uxbridge, United Kingdom, and
Imperial College, London, United Kingdom}
%
%\thanks{The third author was supported in part by EPSRC Grant EP/E022928/1.}

%\date{}

\begin{abstract}
We study the asymptotics in $n$ for $n$-dimensional Toeplitz determinants
whose symbols possess Fisher-Hartwig singularities on a smooth background.
We prove the general non-degenerate asymptotic behavior as conjectured by Basor and Tracy.
We also obtain asymptotics of Hankel determinants on a finite interval
as well as determinants of Toeplitz+Hankel type.
Our analysis is based on a study of the related
system of orthogonal polynomials on the unit circle using the Riemann-Hilbert approach.
\end{abstract}

\maketitle

\section{Introduction}
Let $f(z)$ be a complex-valued function integrable over the unit circle with Fourier coefficients
\[
f_j={1\over 2\pi}\int_0^{2\pi}f(e^{i\theta})e^{-i j\theta}d\theta,\qquad j=0,\pm1,\pm2,\dots
\]
We are interested in the $n$-dimensional Toeplitz determinant with symbol $f(z)$,
\be\la{TD}
D_n(f(z))=\det(f_{j-k})_{j,k=0}^{n-1}.
\ee

In this paper we consider the asymptotics of $D_n(f(z))$ as $n\to\infty$
and of the related orthogonal polynomials as well as the asymptotics of
Hankel, and Toeplitz+Hankel determinants
in the case when the symbol $f(e^{i\theta})$ has a fixed number of
Fisher-Hartwig singularities \ci{FH,L}, i.e.,
when $f(e^{i\theta})$ has the following form on the unit circle $C$:
\be\la{fFH}
f(z)=e^{V(z)} z^{\sum_{j=0}^m \bt_j}
\prod_{j=0}^m  |z-z_j|^{2\al_j}g_{z_j,\bt_j}(z)z_j^{-\bt_j},\qquad z=e^{i\th},\qquad
\theta\in[0,2\pi),
\ee
for some $m=0,1,\dots$,
where
\begin{eqnarray}
&z_j=e^{i\th_j},\quad j=0,\dots,m,\qquad
0=\th_0<\th_1<\cdots<\th_m<2\pi;&\la{z}\\
&g_{z_j,\bt_j}(z)\equiv g_{\bt_j}(z)=
\begin{cases}
e^{i\pi\bt_j}& 0\le\arg z<\th_j\cr
e^{-i\pi\bt_j}& \th_j\le\arg z<2\pi
\end{cases},&\la{g}\\
&\Re\al_j>-1/2,\quad \bt_j\in\bbc,\quad j=0,\dots,m,&
\end{eqnarray}
and $V(e^{i\theta})$ is a sufficiently smooth function on the unit
circle (see below). Here the condition on $\al_j$ insures
integrability. Note that a single Fisher-Hartwig singularity at
$z_j$ consists of a root-type singularity \be\la{za}
|z-z_j|^{2\al_j}=\left|2\sin\frac{\th-\th_j}{2}\right|^{2\al_j} \ee
and a jump $g_{\bt_j}(z)$. A point $z_j$, $j=1,\dots,m$ is included
in (\ref{z}) if and only if either $\al_j\neq 0$ or $\bt_j\neq 0$
(or both); in contrast, we always fix $z_0=1$ even if
$\al_0=\bt_0=0$ (note that $g_{\bt_0}(z)=e^{-i\pi\bt_0})$. Observe
that for each $j=1,\dots,m$, $z^{\beta_j} g_{\beta_j}(z)$ is
continuous at $z=1$, and so for each $j$ each ``beta'' singularity
produces a jump only at the point $z_j$. The factors $z_j^{-\bt_j}$
are singled out to simplify comparisons with existing literature.
Indeed, (\ref{fFH}) with the notation $b(\th)=e^{V(e^{i\theta})}$ is
exactly the symbol considered in
\ci{FH,B,B2,BE1,BEsym,BEnonsym,BE4,BS,BS2,BS3,ES,Ehr,W}. We write
the symbol, however, in a form with $z^{\sum_{j=0}^m \bt_j}$
factored out. The present way of writing $f(z)$ is more natural for
our analysis.

%For a simple example of a symbol of type (\ref{fFH}) consider the
%function $f(e^{i\th})$ equal to a nonzero constant $a$ for
%$0<\th_1\le\th<\th_2<2\pi$, and equal to $1$ on the rest of $C$.
%Then $f(z)=g_{z_1,\bt}(z)g_{z_2,-\bt}(z)$, $a=e^{-2\pi i\bt}$.
A simple example of a symbol of type (\ref{fFH}) is given by
(\ref{example},\ref{tildef}) below. Note that finite order zeros
also give rise to Fisher-Hartwig singularities: for example, if a
sufficiently smooth $f(z)$ has two simple zeros at $0 < \th_1 <
\th_2< 2\pi$, then
$f(z)=e^{V(z)}|z-z_1||z-z_2|g_{z_1,1/2}(z)g_{z_2,-1/2}(z)\left(\frac{z_1}{z_2}
\right)^{-1/2}$ for a suitable $V(z)$.

On the unit circle, $V(z)$ is represented by its Fourier expansion:
\be\la{fourier}
V(z)=\sum_{k=-\infty}^\infty V_k z^k,\qquad
V_k={1\over 2\pi}\int_0^{2\pi}V(e^{i\th})e^{-ki\th}d\th.
\ee

The canonical Wiener-Hopf factorization of $e^{V(z)}$ is
\be\la{WienH}
e^{V(z)}=b_+(z) e^{V_0} b_-(z),\qquad b_+(z)=e^{\sum_{k=1}^\infty V_k z^k},
\qquad b_-(z)=e^{\sum_{k=-\infty}^{-1} V_k z^k}.
\ee

In the case that $\al_j=\bt_j=0$, $f(z)=e^{V(z)}$, the classical strong limit theorem
of Szeg\H o (in its most general form, see, e.g., \cite{Simon}) asserts that as $n\to\infty$,
\be
D_n(f)=\exp\left\{nV_0+\sum_{k=1}^\infty k V_k V_{-k}\right\}(1+o(1)),
\ee
provided $V(z)\in H^{1/2}=\{V=\sum_{k=-\infty}^\infty V_k z^k: \sum_{k=-\infty}^\infty |k||V_k|^2<\infty\}$.

Fisher and Hartwig \cite{FH} were led to single out symbols of type (\ref{fFH})
based on the solution of a variety of specific problems from statistical mechanics,
in particular, the solution of the spontaneous magnetisation problem for the Ising model.
Indeed the square of the magnetisation can be expressed as the limit as $n\to\infty$ of a
Toeplitz determinant $D_n(f)$  (which represents
a 2-spin correlation function at distance $n$ between spins) where the symbol
$f$ is a particular example of (\ref{fFH}) and
has the following properties depending on whether temperature $T$ is lower, equal or higher than
the critical temperature $T_c$:
\begin{itemize}
\item for $T<T_c$, $f$ has no Fisher-Hartwig singularities;
\item for $T=T_c$, $f$ has one singularity at $z_0=1$ with $\al_0=0$, $\bt_0=-1/2$;
\item for $T>T_c$, $f$ has one singularity at $z_0=1$ with $\al_0=0$, $\bt_0=-1$.
\end{itemize}

For $f$ of type (\ref{fFH}), Fisher and Hartwig made a general conjecture in \cite{FH}
about the asymptotic form of $D_n(f)$,
\be\label{FHconj}
D_n(f)\sim E n^\si e^{nV_0},\qquad n\to\infty,
\ee
where $\si=\sum_{j=0}^m (\al_j^2-\bt_j^2)$, and $E$ is a constant depending on $f$.
Considerable effort has been expended in the mathematics and physics communities
in verifying (\ref{FHconj}).

Introduce the seminorm: \be\label{seminorm}
|||\bt|||=\max_{j,k}|\Re\bt_j-\Re\bt_k|, \ee where $1\le j,k\le m$
if $\al_0=\bt_0=0$, and $0\le j,k\le m$ otherwise. If $m=0$, set
$|||\bt|||=0$. Note that in the case of a single singularity, we
always have $|||\bt|||=0$.

The case when $|||\bt|||<1$, i.e., when all $\Re\bt_j$
lie in a single open interval of length 1,
namely $\Re\bt_j\in(q-1/2,q+1/2)$, for an appropriate $q\in\bbr$, has now been essentially settled
(see, however, Remark \ref{smoothness1} below):
In \ci{W}, Widom proved the conjecture when $\Re\al_j>-1/2$, and all
$\bt_j=0$. In \ci{B}, Basor then verified the conjecture when
$\Re\al_j>-1/2$, and $\Re\bt_j=0$.
In \ci{BS},
B\"ottcher and Silbermann established the result in the case that
$|\Re\al_j|<1/2$, $|\Re\bt_j|<1/2$. Finally, in \ci{Ehr},
Ehrhardt verified the conjecture for
$\Re\al_j>-1/2$, $|||\bt|||<1$.
In these papers, the explicit form of $E$ was also established
(see \ci{Ehr} for a review of these and other related results).

\begin{theorem}\la{asTop} (Ehrhardt  \ci{Ehr}).
Let $f(e^{i\theta})$ be defined in (\ref{fFH}), $V(z)$ be $C^\infty$ on the unit circle,
$|||\bt|||<1$, $\Re\al_j>-1/2$, and $\al_j\pm\bt_j\neq -1,-2,\dots$
for $j,k=0,1,\dots,m$. Then as $n\to\infty$,
\begin{multline}\la{asD}
D_n(f)=\exp\left[nV_0+\sum_{k=1}^\infty k V_k V_{-k}\right]
\prod_{j=0}^m b_+(z_j)^{-\al_j+\bt_j}b_-(z_j)^{-\al_j-\bt_j}\\
\times
n^{\sum_{j=0}^m(\al_j^2-\bt_j^2)}\prod_{0\le j<k\le m}
|z_j-z_k|^{2(\bt_j\bt_k-\al_j\al_k)}\left({z_k\over z_j e^{i\pi}}
\right)^{\al_j\bt_k-\al_k\bt_j}\\
\times
\prod_{j=0}^m\frac{G(1+\al_j+\bt_j) G(1+\al_j-\bt_j)}{G(1+2\al_j)}
\left(1+o(1)\right),
\end{multline}
where
$G(x)$ is Barnes' $G$-function. The double product over $j<k$ is set to $1$
if $m=0$.
\end{theorem}

\begin{remark}
The branches in (\ref{asD}) are determined in the natural way as follows:
$b_\pm(z_j)^{-\al_j\pm\bt_j}=\exp\{(-\al_j\pm\bt_j)\sum_{k=1}^\infty V_{\pm k}z^{\pm k}\}$,
$(z_k z_j^{-1} e^{-i\pi})^{\al_j\bt_k-\al_k\bt_j}=\exp\{i(\th_k-\th_j-\pi)(\al_j\bt_k-\al_k\bt_j)\}$,
and the remaining branches are principal.
\end{remark}

\begin{remark}
In the case of  a single singularity,
i.e. when $m=0$ or $m=1$, $\al_0=\bt_0=0$,
the theorem implies that the asymptotics (\ref{asD}) hold for
\be
\Re\al_m>-{1\over 2},\qquad \bt_m\in\bbc,\qquad \al_m\pm\bt_m\neq -1,-2,\dots
\ee
In fact, if there is only one singularity and $V\equiv 0$, an explicit formula is known \ci{BS}
for $D_n(f)$ in terms of the G-functions.
\end{remark}

\begin{remark}
If all $\Re\bt_j\in(-1/2,1/2]$ or all $\Re\bt_j\in[-1/2,1/2)$,
the conditions $\al_j\pm\bt_j\neq -1,-2,\dots$ are satisfied automatically
as $\Re\al_j>-1/2$.
\end{remark}

\begin{remark}\la{degen1}
Since $G(-k)=0$, $k=0,1,\dots$, the formula (\ref{asD}) no longer
represents the leading asymptotics
if $\al_j+\bt_j$ or $\al_j-\bt_j$ is a negative integer for some $j$.
A similar situation arises in Theorem \ref{BT} below if
some representations in $\mathcal{M}$ are degenerate.
These cases can be approached using Lemma \ref{Chr} below, but we do not address
them in the paper.
\end{remark}

\begin{remark}\la{error}
Assume that the function $V(z)$ is analytic. Then the following
can be said about the remainder term (for more results and
details on the error terms see \cite{DIKadd}).
If all $\bt_j=0$, the error term in (\ref{asD}) is of order
$O(n^{-1})$. If there is only one singularity
the error term is also $O(n^{-1})$. In the general case,
the error term depends on the differences $\bt_j-\bt_k$.
Our methods allow us to calculate several asymptotic
terms rather than just the main one presented in
(\ref{asD}) (and also in (\ref{asDgen}) below).
In \cite{DIKadd}, we show that the expansion (\ref{asD}) with analytic $V(z)$
is uniform  in all $\al_j$, $\bt_j$
for $\bt_j$ in compact subsets of the strip $|\Re\bt_j-\Re\bt_k|<1$,
for $\al_j$ in compact
subsets of the half-plane $\Re\al_j>-1/2$, and outside a neighborhood
of the sets $\al_j\pm\bt_j=-1,-2,\dots$. It will be clear below that
given this uniformity, Theorems \ref{asHankel}, \ref{T+H}
also hold uniformly in the same sense, while for  Theorem \ref{BT}
one should replace $\bt_j$ with $\wt\bt_j$ (see below) in the condition
of uniformity.
\end{remark}

\begin{remark}\la{smoothness1}
Theorem \ref{asTop} as proved by Ehrhardt (and as a consequence,
Theorems \ref{BT}, \ref{asHankel}, \ref{T+H} that we proved
below) hold for $C^\infty$ functions $V(z)$ on the unit circle.
In \cite{DIKadd}, we extend Theorem \ref{asTop} to less smooth $V(z)$.
Namely, it is sufficient that the condition
\be\la{Vcond}
\sum_{k=-\infty}^\infty |k|^s |V_k|<\infty
\ee
holds for some $s$ (and hence for all values in $(0,s)$) such that
\be\la{s-main0}
s>
\frac{1+\sum_{j=0}^m\left[(\Im\al_j)^2+(\Re\bt_j)^2\right]}{1-|||\bt|||}.
\ee
In the present work, we show that given Theorem \ref{asTop}
with the condition (\ref{s-main0}) on $V(z)$, Theorems \ref{asHankel}, \ref{T+H}
hold for $V(z)$ under a similar condition with $m$ replaced by $r+1$ and
contributions from $\al_0$, $\al_{r+1}$ appropriately changed,
while Theorem \ref{BT}
holds under the condition (\ref{s-main}) of Remark \ref{smoothness2} below.
The uniformity in $\al$-, $\bt$-parameters will also hold
provided $s$ is taken large enough.
\end{remark}

In \cite{DIKadd}, we give
an independent proof of  Theorem \ref{asTop}, in the spirit
of \cite{D,IK,Kduke}, using a connection of $D_n(f)$ with
the system of polynomials orthogonal with weight $f(z)$ (\ref{fFH}) on the unit circle.
These polynomials also play a central role in the proofs presented here.

It follows, in particular, from Theorem \ref{asTop} that all $D_k(f)\neq 0$, $k=k_0,k_0+1\dots$,
for some sufficiently large $k_0$ if  $\al_j\pm\bt_j\neq -1,-2,\dots$.
Then the polynomials
$\phi_k(z)=\chi_k z^k+\cdots$, $\widehat\phi_k(z)=\chi_k z^{k}+\cdots$
of degree $k$, $k=k_0,k_0+1,\dots$, satisfying
\begin{multline}\la{or0}
{1\over 2\pi}\int_0^{2\pi}\phi_k(z)z^{-j}f(z)d\theta=\chi_k^{-1}\de_{jk},\qquad
{1\over 2\pi}\int_0^{2\pi}\widehat\phi_k(z^{-1})z^j f(z)d\theta=
\chi_k^{-1}\de_{jk},\\
z=e^{i\theta},\qquad j=0,1,\dots,k,
\end{multline}
exist. It is easy to see that they are given by the following expressions:
\be\la{ef1}
\phi_k(z)={1\over\sqrt{D_k D_{k+1}}}
\left|
\begin{matrix}
f_{00}& f_{01}& \cdots & f_{0k}\cr
f_{10}& f_{11}& \cdots & f_{1k}\cr
\vdots & \vdots &  & \vdots \cr
f_{k-1\,0} & f_{k-1\,1} & \cdots & f_{k-1\,k} \cr
1& z& \cdots & z^k
\end{matrix}
\right|,
\ee
\be\la{ef2}
\widehat\phi_k(z^{-1})={1\over\sqrt{D_k D_{k+1}}}
\left|
\begin{matrix}
f_{00}& f_{01}& \cdots & f_{0\,k-1}& 1\cr
f_{10}& f_{11}& \cdots & f_{1\,k-1}& z^{-1}\cr
\vdots & \vdots &  & \vdots & \vdots\cr
f_{k0} & f_{k1} & \cdots & f_{k\,k-1}& z^{-k}
\end{matrix}
\right|,
\ee
where
\[
f_{st}={1\over 2\pi}\int_0^{2\pi}f(z)z^{-(s-t)}d\theta,\quad
s,t=0,1,\dots,k.
\]
We obviously have
\be\la{chiD}
\chi_k=\sqrt{D_k\over D_{k+1}}.
\ee
These polynomials
satisfy a Riemann-Hilbert problem. In Section \ref{RHa}, we solve the problem asymptotically for large $n$ in case of the weight given by (\ref{fFH}) with analytic $V(z)$, thus obtaining the large $n$ asymptotics of the orthogonal polynomials. The main new feature of the
solution is a construction of the local parametrix at the points $z_j$ of Fisher-Hartwig singularities. This parametrix is given in terms of the confluent hypergeometric function (see Proposition \ref{Param}).
A study of the asymptotic behavior of the polynomials
orthogonal on the unit circle was initiated by Szeg\H o \ci{Szego}.
Riemann-Hilbert methods developed within the last 20 years allow us to
find asymptotics of orthogonal polynomials in all regions
of the complex plane (see \ci{Dstrong} and many subsequent works
by many authors).
Such an analysis of the polynomials with an analytic weight on the unit circle
was carried out in \ci{MMS1}, and for the case of a weight with
$\al_j$-singularities but without jumps, in \ci{MMS2}.
We provide, therefore, a generalization of these results.
Here we present only the following statement we will need below
for the analysis of determinants.

\begin{theorem}\la{poly}
Let $f(e^{i\theta})$ be defined in (\ref{fFH}),
$V(z)$ be analytic in a neighborhood of the unit circle,
and $\phi_k(z)=\chi_k z^k+\cdots$, $\widehat\phi_k(z)=\chi_k z^{k}+\cdots$ be
the corresponding polynomials satisfying (\ref{or0}).
Assume that
$|||\bt|||<1$, $\al_j\pm\bt_j\neq -1,-2,\dots$, $j,k=0,1,\dots, m$.
Let
\be\la{de}
\delta=\max_{j,k} n^{2\Re(\bt_j-\bt_k-1)}=n^{2|||\beta|||-2},
\ee
where the indices $j,k=0$ are omitted if $\al_0=\bt_0=0$.

Then
as $n\to\infty$,
\begin{multline}\la{aschi}
\chi_{n-1}^2=\exp\left[ -\int_0^{2\pi}V(e^{i\th}){d\th\over2\pi}\right]\left(
1-{1\over n}\sum_{k=0}^m (\al_k^2-\bt_k^2)\right.\\
+\left.
\sum_{j=0}^m\sum_{k\neq j}{z_k\over z_j-z_k}
\left({z_j\over z_k}\right)^n n^{2(\bt_k-\bt_j-1)}{\nu_j\over\nu_k}
{\Ga(1+\al_j+\bt_j)\Ga(1+\al_k-\bt_k)\over
\Ga(\al_j-\bt_j)\Ga(\al_k+\bt_k)}\frac{b_+(z_j)b_-(z_k)}{b_-(z_j)b_+(z_k)}
\right.\\
\left.
+ O(\de^2)+O(\de/n)\right),
\end{multline}
where
\be\la{nu}
\nu_j=\exp\left\{-i\pi\left(\sum_{p=0}^{j-1}\al_p-
\sum_{p=j+1}^m\al_p\right)\right\}
\prod_{p\neq j}\left({z_j\over z_p}\right)^{\al_p}
|z_j-z_p|^{2\bt_p}.
\ee
Under the same conditions,
\be\la{asphi}
\phi_n(0)=\chi_n \left(
\sum_{j=0}^m n^{-2\bt_j-1} z_j^n \nu_j
{\Gamma(1+\al_j+\bt_j)\over \Gamma(\al_j-\bt_j)}
\frac{b_+(z_j)}{b_-(z_j)}+
O\left(\left[\de+{1\over n}\right]\max_k{n^{-2\Re\bt_k}\over n}
\right)\right),
\ee
\be\la{ashatphi}
\widehat\phi_n(0)=\chi_n
\left(
\sum_{j=0}^m n^{2\bt_j-1} z_j^{-n} \nu_j^{-1}
{\Gamma(1+\al_j-\bt_j)\over \Gamma(\al_j+\bt_j)}
\frac{b_-(z_j)}{b_+(z_j)}+
O\left(\left[\de+{1\over n}\right]\max_k{n^{2\Re\bt_k}\over n}
\right)\right).
\ee
\end{theorem}

\begin{remark}\la{17}
The error terms here are uniform and differentiable in all $\al_j$, $\bt_j$
for $\bt_j$ in compact subsets of the strip $|\Re\bt_j-\Re\bt_k|<1$,
for $\al_j$ in compact
subsets of the half-plane $\Re\al_j>-1/2$, and outside a neighborhood
of the sets $\al_j\pm\bt_j=-1,-2,\dots$. If $\al_j+\bt_j=0$ or $\al_j-\bt_j=0$
for some $j$, the corresponding terms in the above formulae vanish.
\end{remark}
\begin{remark}
Note that the terms with $n^{2(\bt_k-\bt_j-1)}$ in (\ref{aschi}) become
larger in absolute value than the $1/n$ term for $|||\bt|||>1/2$.
\end{remark}
\begin{remark}
With changes to the error estimates,
this theorem can be generalized to sufficiently smooth $V(z)$ using
(\ref{chiD}), a well-known representation for orthogonal polynomials
as multiple integrals, and
similar arguments to those we give in Section \ref{secBTb} below.
\end{remark}

Our first task in this paper is to extend the asymptotic formula for $D_n(f)$
to arbitrary $\bt_j\in\bbc$,
i.e. for the case when not all $\Re\bt_j$'s lie in a single interval of
length less than 1, in other words for
\[
|||\bt|||\ge 1.
\]
We know from examples
(see, e.g., \ci{BS,BT,Ehr}) that in general, the formula (\ref{asD}) breaks down.
Obviously, the general case can be reduced to $\Re\bt_j\in(q-1/2,q+1/2)$ by adding
integers to $\bt_j$. Then, apart from a constant factor, the only change in $f(z)$ is
multiplication with $z^\ell$, $\ell\in\bbz$. However,
as we show in Lemma \ref{Chr}, the determinants $D_n(f(z))$ and $D_n(z^\ell f(z))$ are
simply related.
They differ just by a factor which involves $\chi_k$, $\phi_k(0)$, $\widehat\phi_k(0)$ for
large $k$ (these quantities are given by Theorem \ref{poly}), as well as the
derivatives of the orthogonal polynomials at $0$.
The derivatives can be calculated similarly to $\phi_k(0)$, $\widehat\phi_k(0)$.
Thus it is easy to obtain the general asymptotic formula for
$D_n(f(z))$.
However, this formula is implicit in the sense that one still
needs to separate the main asymptotic term from the others: e.g.,
if the dimension $\ell$ of $F_n$ in (\ref{213}) is larger than the
number of the leading-order terms in (\ref{asphi}), the obvious
candidate for the leading order in $F_n$ vanishes (this
is not the case in the simplest situation given by Theorem
\ref{BT1}). We resolve this problem below.

Following \ci{BT,Ehr}, we define a {\it FH-representation}
of a symbol. Namely, for $f(z)$ given by (\ref{fFH}) replace
$\bt_j$ by $\bt_j+n_j$, $n_j\in\bbz$
if $z_j$ is a singularity (i.e., if either $\bt_j\neq 0$ or
$\al_j\neq 0$ or both).
The integers $n_j$ are arbitrary subject to the condition
$\sum_{j=0}^m n_j=0$.
In a slightly different notation from \ci{BT,Ehr}, we call the
resulting function $f(z;n_0,\dots,n_m)$ a FH-representation of $f(z)$.
(The original $f(z)$ is also a FH-representation corresponding to
$n_0=\cdots=n_m=0$.)
Obviously, all FH-representations of $f(z)$ differ only by
multiplicative constants.
We have
\be
f(z)=\prod_{j=0}^m z_j^{n_j}\times f(z;n_0,\dots,n_m).
\ee
We are interested in the FH-representations (characterized by
$(n_j)_{j=0}^m$) of $f$ such that
$\sum_{j=0}^m (\Re\bt_j+n_j)^2$ is minimal. There is a finite number
of such FH-representations and we provide an algorithm for finding
them explicitly (see the proof of Lemma \ref{repr} below). We
denote the set of such FH-representations by
$\mathcal{M}$. Furthermore, we call a FH-representation
degenerate if $\al_j+(\bt_j+n_j)$ or $\al_j-(\bt_j+n_j)$ is a negative
integer for some $j$. We call $\mathcal{M}$ non-degenerate if it contains
no degenerate FH-representations.
The set $\mathcal{M}$ can be characterized as follows. For a given
$\bt=(\bt_0,\dots,\bt_m)$ let us call
\[
O_\bt=\{\widehat\bt: \widehat\bt_j=\bt_j+n_j, \sum_{j=0}^m n_j=0\}
\]
the {\it orbit} of $\bt$ (if $\al_0=\bt_0=0$ then we always fix
$n_0=0$). In other words, it is the set of $\widehat\bt$
corresponding to all the FH-representations of $f$. For any
$\widehat\bt$ in $O_\bt$, we define $|||\widehat\bt|||=
\max_{j,k}|\Re\widehat\bt_j-\Re\widehat\bt_k|$ with the range of
indicies $j$, $k$ fixed to be the same as for $|||\bt|||$
(\ref{seminorm}). We have

\begin{lemma}\la{repr}
There exist only the following 2 mutually exclusive possibilities.
\begin{itemize}
\item $\exists \widehat\bt\in O_\bt$ such that $|||\widehat\bt|||<1$. Then such $\widehat\bt$
is unique and it is the unique element of
$\mathcal{M}=\{\widehat\bt\}$.
\item $\exists \widehat\bt\in O_\bt$ such that $|||\widehat\bt|||=1$. Then there are at least 2
such $\widehat\bt$'s and all of them are obtained from each other by
a repeated application of the following rule:
add $1$ to a $\widehat\bt_j$ with the smallest real part
and subtract $1$ from a $\widehat\bt_j$ with the largest. Moreover,
$\mathcal{M}=\{\widehat\bt\in O_\bt:|||\widehat\bt|||=1\}$.
\end{itemize}
\end{lemma}

\begin{proof}
Suppose that the seminorm $|||\bt|||>1$. Then, writing
$\bt^{(1)}_s=\bt_s+1$,  $\bt^{(1)}_t=\bt_t-1$, and
$\bt^{(1)}_j=\bt_j$ if $j\neq s,t$, where $\bt_s$ is one of the
beta-parameters with $\Re\bt_s=\min_j\Re\bt_j$,  $\bt_t$  is one of
the beta-parameters with $\Re\bt_t=\max_j\Re\bt_j$, we see that
$|||\bt^{(1)}|||\le|||\bt|||$, and $f$ corresponding to $\bt^{(1)}$
is a FH-representation. After a finite number, say $r$, of such
transformations we reduce an arbitrary set of $\bt_j$ to the
situation for which either $|||\bt^{(r)}|||<1$ or
$|||\bt^{(r)}|||=1$. Note that further transformations do not change
the seminorm in the second case, while in the first case the
seminorm oscillates periodically taking two values,
$|||\bt^{(r)}|||$ and $2-|||\bt^{(r)}|||$. Thus all the symbols of
type (\ref{fFH}) belong to two distinct classes: the first, for
which $|||\bt^{(r)}|||<1$, and the second, for which
$|||\bt^{(r)}|||=1$. For symbols of the first class, $\mathcal{M}$
has only one member with beta-parameters $\bt^{(r)}$. Indeed,
writing $b_j=\Re\bt_j$, if $-1/2<b_j^{(r)}-q< 1/2$ for some
$q\in\bbr$ and all $j$, then for any $(k_j)_{j=0}^m$ such that
$\sum_{j=0}^m k_j=0$ and not all $k_j$ are zero, we have
\be\la{ineq} \sum_{j=0}^m (b_j^{(r)}+k_j)^2=\sum_{j=0}^m
(b_j^{(r)})^2+ 2\sum_{j=0}^m(b_j^{(r)}-q)k_j+ \sum_{j=0}^m k_j^2
>\sum_{j=0}^m (b_j^{(r)})^2+\sum_{j=0}^m k_j^2-|k_j|\ge
\sum_{j=0}^m (b_j^{(r)})^2,
\ee
where the first inequality is strict as at least two $k_j\neq 0$.
For symbols of the second class, we can find $q\in\bbr$ such that
$-1/2\le b_j^{(r)}-q\le 1/2$ for all $j$. Equation (\ref{ineq}) in this
case holds with ``$>$'' sign replaced by ``$\ge$''. Clearly, there are several
FH-representations in $\mathcal{M}$ in this case (they correspond to
the equalities in (\ref{ineq})) and
adding $1$ to one of $\bt^{(r)}_s$ with $b^{(r)}_s=\min_j b^{(r)}_j=
q-1/2$
while subtracting $1$ from one of $\bt^{(r)}_t$ with
$b^{(r)}_t=\max_j b^{(r)}_j=q+1/2$
provides the way to find all of them.

A simple explicit sufficient, but obviously not necessary, condition
for $\mathcal{M}$ to have only one member is that all $\Re\bt_j\mod 1$
be different.
\end{proof}

In Section \ref{secBT}, we prove

\begin{theorem}\la{BT}
Let $f(z)$ be given in (\ref{fFH}), $\Re\al_j>-1/2$, $\bt_j\in\bbc$, $j=0,1,\dots,m$.
Let $\mathcal{M}$ be non-degenerate.
Then, as $n\to\infty$,
\be\la{asDgen}
D_n(f)=\sum\left(\prod_{j=0}^m z_j^{n_j}\right)^n
\mathcal{R}(f(z;n_0,\dots,n_m))(1+o(1)),
\ee
where the sum is over all FH-representations in $\mathcal{M}$.
Each $\mathcal{R}(f(z;n_0,\dots,n_m))$ stands for the right-hand side
of the formula (\ref{asD}), without the error term,
corresponding to $f(z;n_0,\dots,n_m)$.
\end{theorem}

\begin{remark}
This theorem was conjectured by Basor and Tracy \ci{BT} based on an explicit example:
see Example \ref{example} below.
The case when the FH-representation minimizing
$\sum_{j=0}^m (\Re\bt_j+n_j)^2$ is unique, i.e. there is only
one term in the sum (\ref{asDgen}), was proved by Ehrhardt \ci{Ehr}.
Note that this case is exactly the first possibility of Lemma \ref{repr}.
Thus, Theorem \ref{BT} in this case follows from Theorem
\ref{asTop} applied to this FH-representation.
\end{remark}

\begin{remark}\la{smoothness2}
This theorem relies on Theorem \ref{asTop} and therefore requires $V(z)$ to be
$C^\infty$ on the unit circle. As remarked above, we prove in \ci{DIKadd} that
Theorem \ref{asTop} holds in fact under the condition (\ref{s-main0}).
It then follows (see Section \ref{secBT}) that,
if $\mathcal{M}$ has several members, Theorem \ref{BT} holds for any
\be\la{s-main}
s>
\frac{1+\sum_{j=0}^m\left[(\Im\al_j)^2+
\max\left\{(\Re\wt\bt_j)^2,(\Re\bt^{(r)}_j)^2\right\}\right]}{1-|||\wt\bt|||},
\ee
where $\wt\bt_j$ are obtained from $\bt^{(r)}_j$ (see the proof of Lemma \ref{repr})
by subtracting $1$ from all  $\bt^{(r)}_j$
with the maximal real part and leaving the rest unchanged. The number $|||\wt\bt|||<1$
is defined the same way as $|||\widehat\bt|||$, $|||\bt|||$ above.
\end{remark}

\begin{remark}
The situation when all $\al_j\pm\bt_j$ are nonnegative integers, which was
considered by B\"ottcher and Silbermann in \ci{BS2}, is a particular case
of the above theorem.
\end{remark}

\begin{remark}\la{degen2}
The case when {\it all} the FH-representations of $f$ are degenerate
(not only those in $\mathcal{M}$) was considered by Ehrhardt \ci{Ehr} who found
that in this case $D_n(f)=O(e^{nV_0}n^r)$, where $r$ is any real number.
We can reproduce this result by our methods but do not present it here.
\end{remark}

We will now discuss a simple particular case of Theorem \ref{BT}
and present a direct independent proof in this case.

\begin{theorem}[A particular case of Theorem \ref{BT}]\la{BT1}
 Let the symbol $f^{\pm}(z)$ be obtained from $f(z)$ (\ref{fFH})
by replacing one $\bt_{j_0}$ with $\bt_{j_0}\pm 1$ for some fixed
$0\le j_0\le m$. Let $\Re\al_j>-{1\over 2}$, $\Re\bt_j\in
(-1/2,1/2]$, $j=0,1,\dots,m$. Then, for sufficiently large $n$,
\be
D_n(f^{+}(z))=z_{j_0}^{-n} {\phi_n(0)\over\chi_n}D_n(f(z)),\qquad
D_n(f^{-}(z))=z_{j_0}^n {\widehat\phi_n(0)\over\chi_n}D_n(f(z)).
\ee
These formulae together with
(\ref{asphi},\ref{ashatphi},\ref{aschi},\ref{asD}) yield the
following asymptotic description of $D_n(f^{\pm})$. Let there be
more than one singular point $z_j$ and all $\al_j\pm\bt_j\neq 0$.
For $f^{+}(z)$, let $\bt_{j_p}$, $p=1,\dots,s$
 be such that they have the same real part which is strictly less
than the real parts of all the other $\bt_j$ (if any), i.e.
$\Re\bt_{j_1}=\cdots=\Re\bt_{j_s}<\min_{j\neq j_1,\dots,j_s}\Re\bt_j$.
For $f^{-}(z)$ let one $\bt_{j_p}$, $p=1,\dots,s$ be
such that  $\Re\bt_{j_1}=\cdots=\Re\bt_{j_s}>\max_{j\neq j_1,\dots,j_s}
\Re\bt_j$.
Then the asymptotics of $D_n(f^{\pm})$ are given by the following:
\be\la{asf}
D_n(f^{+})=z_{j_0}^{-n}\sum_{p=1}^s z_{j_p}^n\mathcal{R}_{j_p,+}(1+o(1)),\qquad
D_n(f^{-})=z_{j_0}^{n}\sum_{p=1}^s z_{j_p}^{-n}\mathcal{R}_{j_p,-}(1+o(1)),
\ee
where $\mathcal{R}_{j,\pm}$ is the right-hand side of (\ref{asD})
(without the error term)
in which $\bt_j$ is replaced by $\bt_j\pm 1$, respectively.
\end{theorem}

\begin{proof}
For simplicity, we present the proof only for $V(z)$ analytic in
a neighborhood of the unit circle.
Consider the case of $f^{-}(z)$. It corresponds to one of the $\bt_j$ shifted
inside the interval $(-3/2,-1/2]$. Since
\[
z^{\sum_{j=0}^m\bt_j-1}=z^{-1}z^{\sum_{j=0}^m\bt_j},\qquad g_{\bt_{j_0}-1}(z)=
-g_{\bt_{j_0}}(z),\qquad z_{j_0}^{-\bt_{j_0}+1}=z_{j_0}z_{j_0}^{-\bt_{j_0}},
\]
we see that
\[
f^{-}(z)=-z_{j_0}z^{-1}f(z).
\]
Therefore, using the identity (\ref{DD-1}) below,
we obtain
\[
D_n(f^{-}(z))=(-z_{j_0})^n D_n(z^{-1}f(z))=
z_{j_0}^n {\widehat\phi_n(0)\over\chi_n}D_n(f(z)).
\]
If, for some $j_1$, $j_2$, $\dots$, $j_s$, we have that
$\Re\bt_{j_1}=\cdots=\Re\bt_{j_s}>\max_{j\neq j_1,\dots,j_s}\Re\bt_j$,
then we see from
(\ref{ashatphi}) that only the addends with $n^{2\bt_{j_1}-1}$, $\dots$,
$n^{2\bt_{j_s}-1}$ give  contributions
to the main asymptotic term of $D_n(f^{-}(z))$. Using Theorem \ref{asTop} for $D_n(f(z))$
and the relation
$G(1+x)=\Ga(x) G(x)$, we obtain the formula (\ref{asf}) for $D_n(f^{-}(z))$.
The case of $f^{+}(z)$ is similar.
\end{proof}

\begin{example}\la{examplesection}
In \ci{BT} Basor and Tracy considered a simple example of a symbol
of type (\ref{fFH}) for which the asymptotics of the determinant can
be computed directly, but are very different from (\ref{asD}). Up to
a constant, the symbol is
\be\la{example}
f^{(BT)}(e^{i\th})=
\begin{cases}
-i,& 0<\th<\pi\cr
i,& \pi<\th<2\pi
\end{cases}.
\ee
We can represent $f^{(BT)}$ as a symbol with $\bt$-singularities
$\bt_0=1/2$, $\bt_1=-1/2$ at the points $z_0=1$ and $z_1=-1$, respectively:
\be\la{tildef}
f^{(BT)}(z)=g_{1,1/2}(z)g_{-1,-1/2}(z)e^{i\pi/2}
\ee
We see that $f^{(BT)}(z)=f^{-}(z)$ and $j_0=1$.
Therefore by the first part of Theorem \ref{BT1}, we have
\[
D_n(f^{(BT)}(z))=(-1)^n{\widehat\phi_n(0)\over\chi_n}D_n(f(z)),
\]
where $\phi_n(z)$, $\chi_n$, $D_n(f(z))$ correspond to $f(z)$ given by
(\ref{fFH}) with $m=1$, $z_0=1$, $z_1=e^{i\pi}$, $\bt_0=\bt_1=1/2$,
$\al_0=\al_1=0$.

Observing that $s=2$, $j_1=j_0=1$ and $j_2=0$ and
using (\ref{asf}) we obtain
\[
D_n(f^{(BT)}(z))=(-1)^n((-1)^n\mathcal{R}_{1,-}+\mathcal{R}_{0,-}).
\]
Since $\mathcal{R}_{1,-}=\mathcal{R}_{0,-}=(2n)^{-1/2}G(1/2)^2 G(3/2)^2(1+o(1))$, we
obtain
\be
D_n(f^{(BT)}(z))=\frac{1+(-1)^n}{2}\sqrt{2\over n} G(1/2)^2 G(3/2)^2(1+o(1)),
\ee
which is the answer found in \ci{BT}.

As noted by Basor and Tracy, $f^{(BT)}(z)$
has a different FH-representation of type (\ref{fFH}), namely,
with $\bt_0=-1/2$, $\bt_1=1/2$, and we can write
\be\la{f2}
f^{(BT)}(z)=-g_{1,-1/2}(z)g_{-1,1/2}(z)e^{-i\pi/2}.
\ee
This fact was the origin of their conjecture. In the notation of Theorem \ref{BT},
the symbol (\ref{tildef}) has the two FH-representations minimizing
$\sum_{j=0}^1 (\Re\bt_j+n_j)^2$, one with $n_0=n_1=0$ and the other with $n_0=-1$, $n_1=1$.
\end{example}

Note that in the case $\sum_{j=0}^m\bt_j=0$
we can always assume that $\Re\bt_j\in[-1/2,1/2]$.
The beta-singularities then are just piece-wise constant (step-like)
functions.
This case is relevant for our next result, which is on Hankel determinants.

Let $w(x)$ be an integrable complex-valued function on the interval $[-1,1]$.
Then the Hankel determinant with symbol $w(x)$ is given by
\be
D_n(w(x))=\det\left(\int_{-1}^1 x^{j+k}
  w(x)dx\right)_{j,k=0}^{n-1}.
\ee
Define $w(x)$ for a fixed $r=0,1,\dots$ as follows:
\begin{multline}\la{wFH}
w(x)=e^{U(x)}\prod_{j=0}^{r+1}|x-\lb_j|^{2\al_j}\om_j(x)\\
1=\lb_0>\lb_1>\cdots>\lb_{r+1}=-1,\qquad
\om_j(x)=
\begin{cases}
e^{i\pi\bt_j}& \Re x\le\lb_j\cr
e^{-i\pi\bt_j}& \Re x>\lb_j
\end{cases},\qquad
\Re\bt_j\in(-1/2,1/2],\\
\bt_0=\bt_{r+1}=0,\qquad \Re\al_j>-{1\over 2}, \qquad j=0,1,\dots,r+1.
\end{multline}
where $U(x)$ is a sufficiently smooth function on the interval $[-1,1]$.
Note that we set $\bt_0=\bt_{r+1}=0$, $\Re\bt_j\in(-1/2,1/2]$
without loss of generality as the
functions $\om_0(x)$, $\om_{r+1}(x)$ are just constants on $(-1,1)$, and
$\om_j(x;\bt_j+k_j)=(-1)^{k_j}\om(x;\bt_j)$, $k_j\in\bbz$.

In Section \ref{Hankel}, we prove
\begin{theorem}\la{asHankel}
Let $w(x)$ be defined as in (\ref{wFH}) with $\Re\bt_j\in\left(-{1\over 2},
{1\over 2}\right)$, $j=1,2,\dots,r$.
Then as $n\to\infty$,
\begin{multline}\la{asDH}
D_n(w)=
D_n(1)e^{\left[(n+\al_0+\al_{r+1})V_0-\al_0 V(1)-\al_{r+1}V(-1)+
{1\over 2}\sum_{k=1}^\infty k V_k^2\right]}\\
\times
\prod_{j=1}^r b_+(z_j)^{-\al_j-\bt_j}b_-(z_j)^{-\al_j+\bt_j}\times
e^{\left[2i(n+A)\sum_{j=1}^r\bt_j\arcsin\lb_j+
i\pi\sum_{0\le j<k\le r+1}(\al_j\bt_k-\al_k\bt_j)
\right]}\\
\times
4^{-\left(An+\al_0^2+\al_{r+1}^2+\sum_{0\le j<k\le r+1}\al_j\al_k+\sum_{j=1}^r\bt_j^2
\right)}
(2\pi)^{\al_0+\al_{r+1}}
n^{2(\al_0^2+\al_{r+1}^2)+\sum_{j=1}^r(\al_j^2-\bt_j^2)}\\
\times
\prod_{0\le j<k\le r+1}|\lb_j-\lb_k|^{-2(\al_j\al_k+\bt_j\bt_k)}
\left|\lb_j\lb_k-1+\sqrt{(1-\lb_j^2)(1-\lb_k^2)}\right|^{2\bt_j\bt_k}\\
\times
{1\over G(1+2\al_0)G(1+2\al_{r+1})}
\prod_{j=1}^r(1-\lb_j^2)^{-(\al_j^2+\bt_j^2)/2}
\frac{G(1+\al_j+\bt_j) G(1+\al_j-\bt_j)}{G(1+2\al_j)}
\left(1+o(1)\right),
\end{multline}
where $A=\sum_{k=0}^{r+1}\al_k$,
$V(e^{i\th})=U(\cos\th)$, $z_j=e^{i\th_j}$, $\lb_j=\cos\th_j$, $j=0,\dots,r+1$,
and the functions $b_\pm(z)$ are defined in (\ref{WienH}).
\end{theorem}

\begin{remark}
$D_n(1)$ is an explicitly computable determinant, cf. \ci{Warc}, related to the Legendre polynomials
(it can also be written as a Selberg integral)
\be
D_n(1)=2^{n^2}\prod_{k=0}^{n-1}
\frac{k!^3}{(n+k)!}
=\frac{\pi^{n+1/2}G(1/2)^2}{2^{n(n-1)}n^{1/4}}\left(1+o(1)\right).
\ee
\end{remark}

To prove Theorem \ref{asHankel} we use the fact that
$w(x)$ can be generated by a particular class of functions
$f(z)$ given by (\ref{fFH}). Namely, we can find an {\it even} function $f$ of $\th$
($f(e^{i\th})=f(e^{-i\th})$, $\th\in [0,2\pi)$)
such that
\be\la{wf}
w(x)={f(e^{i\th})\over|\sin\th|},\qquad x=\cos\th,\quad x\in[-1,1].
\ee
We must have (see Section \ref{Hankel} below) that $m=2r+1$, $\th_0=0$,
$\th_{r+1}=\pi$,
$\th_{m+1-j}=2\pi-\th_j$, $j=1,\dots r$. If we denote the beta-parameters of
$f(z)$ by $\wt\bt_j$, we obtain $\wt\bt_0=\wt\bt_{r+1}=0$,
$\wt\bt_j=-\wt\bt_{m+1-j}=-\bt_j$, $j=1,\dots,r$.
In particular, $\sum_{j=0}^m\wt\bt_j=0$ as remarked above.

In Section \ref{Hankel} we obtain  Theorem \ref{asHankel} from Theorem \ref{asTop} and
the asymptotics for the orthogonal polynomials on the unit circle with weight $f(z)$
using the following connection between Hankel and Toeplitz determinants
established by Theorem \ref{HT} below:
\be\la{131}
D_n(w(x))^2={\pi^{2n}\over 4^{(n-1)^2}}
{(\chi_{2n}+\phi_{2n}(0))^2\over\phi_{2n}(1)\phi_{2n}(-1)}D_{2n}(f(z)),
\ee
where $w(x)$ and $f(z)$ are related by (\ref{wf}).

\begin{remark}\la{b12}
Asymptotics of a Hankel determinant when some (or all) of $\bt_j$ have the real part
$1/2$ can be easily obtained. For the corresponding $f(z)$ this implies that
certain $\Re\wt\bt_j=-1/2$ and $\Re\wt\bt_{m+1-j}=1/2$ and the rest
$\Re\wt\bt_k\in(-1/2,1/2)$. Thus, Theorem \ref{BT} can be used to estimate
$D_{2n}(f(z))$. For the asymptotics of $\phi_{2n}(z)$ in this case we need
an additional ``correction'' $R_1$ term (given by (\ref{plem}) below) which is now
$O(n^{-2\wt\bt_j-1})=O(1)$.
\end{remark}

\begin{remark}
One can obtain the asymptotics of the polynomials orthogonal
on the interval $[-1,1]$ with weight (\ref{wFH}) by using our
results for the polynomials $\phi_k(z)$ orthogonal with the
corresponding even weight on the unit circle
and a Szeg\H o relation (Lemma \ref{Pszego} below) which maps
the latter polynomials to the former ones.
\end{remark}

\begin{remark}
Asymptotics for a subset of symbols (\ref{wFH}) which satisfy a symmetry condition
and have a certain behaviour at the end-points $\pm 1$ were found
by Basor and Ehrhardt in \ci{BE1}. They use relations between
Hankel and Toeplitz determinants
which are less general than (\ref{131}) but do not involve polynomials.
For some other related results, see \ci{Geronimo,KVA}.
\end{remark}

Our final task is to present asymptotics for the so-called Toeplitz+Hankel
determinants.
We consider the four most important ones appearing in the theory of classical groups and
its applications to random matrices and statistical mechanics
(see, e.g., \ci{BR,FF,KM}) defined in terms of
the Fourier coefficients of an even $f$ (evenness implies
the matrices are symmetric) as follows:
\be\la{intT+H}
\det (f_{j-k}+f_{j+k})_{j,k=0}^{n-1},\quad
\det (f_{j-k}-f_{j+k+2})_{j,k=0}^{n-1},\quad
\det (f_{j-k}\pm f_{j+k+1})_{j,k=0}^{n-1}.
\ee
There are simple relations \ci{Weyl,J2,BR} between the determinants (\ref{intT+H})
and Hankel determinants on $[-1,1]$ with added singularities at the end-points.
These are summarized in Lemma \ref{HTH} below.
It is easily seen that if $f(z)$ is an (even) function of type (\ref{fFH})
then the corresponding symbols of Hankel determinants belong to
the class (\ref{wFH}). Thus a straightforward combination of Lemma \ref{HTH} and
Theorem \ref{asHankel} (aided by formulae of Section \ref{Hankel})
gives the following

\begin{theorem}\la{T+H}
Let $f(z)$ be defined in (\ref{fFH}) with the condition
$f(e^{i\th})=f(e^{-i\th})$. Let $\th_{r+1}=\pi$ and
$\Re\bt_j\in\left(-{1\over 2},{1\over2}\right)$,
$j=1,2,\dots,r$, $\bt_0=\bt_{r+1}=0$.
Then as $n\to\infty$,
\begin{multline}\la{TH}
D_n^{\mathrm{T+H}}=
e^{nV_0+{1\over 2}\left[(\al_0+\al_{r+1}+s+t)V_0-(\al_0+s)V(1)-(\al_{r+1}+t)V(-1)+
\sum_{k=1}^\infty k V_k^2\right]}\\
\times
\prod_{j=1}^r b_+(z_j)^{-\al_j+\bt_j}b_-(z_j)^{-\al_j-\bt_j}\times
e^{-i\pi\left[\left\{\al_0+s+\sum_{j=1}^r\al_j\right\}\sum_{j=1}^r\bt_j+
\sum_{1\le j<k\le r}(\al_j\bt_k-\al_k\bt_j)
\right]}\\
\times
2^{(1-s-t)n+p+\sum_{j=1}^r(\al_j^2-\bt_j^2)-{1\over2}(\al_0+\al_{r+1}+s+t)^2
+{1\over2}(\al_0+\al_{r+1}+s+t)}
n^{{1\over2}(\al_0^2+\al_{r+1}^2)+\al_0 s+\al_{r+1}t+
\sum_{j=1}^r(\al_j^2-\bt_j^2)}\\
\times
\prod_{1\le j<k\le r}|z_j-z_k|^{-2(\al_j\al_k-\bt_j\bt_k)}
|z_j-z_k^{-1}|^{-2(\al_j\al_k+\bt_j\bt_k)}\\
\times
\prod_{j=1}^r z_j^{2\wt A\bt_j}|1-z_j^2|^{-(\al_j^2+\bt_j^2)}
|1-z_j|^{-2\al_j(\al_0+s)}|1+z_j|^{-2\al_j(\al_{r+1}+t)}\\
\times
\frac{\pi^{{1\over2}(\al_0+\al_{r+1}+s+t+1)}G(1/2)^2}
{G(1+\al_0+s)G(1+\al_{r+1}+t)}
\prod_{j=1}^r\frac{G(1+\al_j+\bt_j) G(1+\al_j-\bt_j)}{G(1+2\al_j)}
\left(1+o(1)\right),
\end{multline}
where $\wt A={1\over2}(\al_0+\al_{r+1}+s+t)+\sum_{j=1}^r\al_j$ and
\begin{align}
&D_n^{\mathrm{T+H}}=\det (f_{j-k}+f_{j+k})_{j,k=0}^{n-1},\qquad
\mbox{with}\quad p=-2n+2,\quad s=t=-{1\over 2}\\
&D_n^{\mathrm{T+H}}=\det (f_{j-k}-f_{j+k+2})_{j,k=0}^{n-1},\qquad
\mbox{with}\quad p=0,\quad s=t={1\over 2}\\
&D_n^{\mathrm{T+H}}=\det (f_{j-k}\pm f_{j+k+1})_{j,k=0}^{n-1},\qquad
\mbox{with}\quad p=-n,\quad s=\mp{1\over 2},\quad t=\pm{1\over 2}.
\end{align}
\end{theorem}

\begin{remark}
For the case $\Re\bt_j=1/2$ see Remark \ref{b12} above.
\end{remark}

\begin{remark}
For the determinant $\det (f_{j-k}+f_{j+k+1})_{j,k=0}^{n-1}$ in the case
when the symbol has no $\al$ singularities at $z=\pm 1$ and $|\Re\bt_j|<1/2$,
the asymptotics
were obtained in \ci{BEsym} (see also \ci{BEnonsym} if $f$ is non-even, $\al_j=0$).
Note that for symbols without singularities, i.e. for $f(z)=e^{V(z)}$,
the asymptotics of all the above Toeplitz+Hankel determinants
(and related more general ones) were found recently in \ci{BE4}.
\end{remark}

The results in this paper make it possible to justify various asymptotic formulae that
were obtained previously in the literature on the basis of
the Basor-Tracy conjecture and conjectures on Hankel and Toeplitz+Hankel determinants.
One such application arises in the framework of the
random matrix approach in the theory of the Riemann zeta-function and other $L$-functions
(see \cite{Keating1} for a recent survey). Define
\begin{equation}\label{fL}
\phi(z)=\left|2\sin{\theta\over 2}\right|^{2k}e^{V(z)},\quad k\in {\mathbb N},
\end{equation}
where
\[
V(e^{i\theta})=2k\left\{\int_1^e u(y)\left(\sum_{j=-\infty}^\infty
\mathrm{Ci}(|\theta+2\pi j|\ln y\ln X)dy\right)-\ln\left|2\sin{\theta\over 2}\right|\right\},
\]
\[
\mathrm{Ci}(z)=-\int_z^\infty\frac{\cos t}{t}dt,
\]
and  $u(y)$ is a smooth nonnegative function supported on $[e^{1-1/X},e]$ and of total mass one.
Consider the following average over the orthogonal  group $SO(2n)$,
\begin{equation}\label{so}
E_{SO(2n)}\left(\prod_{j=1}^n \phi(e^{i\theta_j})\right).
\end{equation}
This object was introduced  by Bui and  Keating  in  \cite{BK} (following \cite{GHK})
as the random matrix counterpart of a key term contributing to
the mean values of certain  Dirichlet  $L$-functions in
the Katz-Sarnak orthogonal family. The issue is the large $n$ and large $X$ behavior
of this average.  The question can be resolved with the help of Theorem \ref{T+H}.
Indeed, one can observe   that
\[
E_{SO(2n)}\left(\prod_{j=1}^n \phi(e^{i\theta_j})\right) =
\frac{1}{2}\det (\phi_{j-k}+\phi_{j+k})_{j,k=0}^{n-1},
\]
and that symbol (\ref{fL}) is of Fisher-Hartwig type with a single
$\alpha$-singularity at $z_0=1$, and $\alpha_0=k$. A direct
application of Theorem \ref{T+H}  leads then, for $X$ large, to the
following asymptotic behavior as $n\to\infty$ of the average
(\ref{so}),
\begin{equation}\label{orth}
E_{SO(2n)}\left(\prod_{j=1}^n \phi(e^{i\theta_j})\right)\sim
G(1+k)\left(\frac{\Gamma(1+2k)} {G(1+2k)\Gamma(1+k)} \right)^{1/2}
\left(\frac{ 2n}{e^\gamma\ln X}\right)^{k(k-1)/2},
\end{equation}
where $\gamma$ is Euler's constant.
Formula (\ref{orth}) is precisely the asymptotic form conjectured by Bui and Keating in \cite{BK}.
In a similar way, Theorem \ref{T+H}  provides
a justification of an analogous conjecture of
Bui and Keating in \cite{BK} concerning  the average of the same product
$\prod_{j=1}^n \phi(e^{i\theta_j})$  over the symplectic group. In the context of the random
matrix approach in number theory \cite{Keating1}, this means that
Theorem \ref{T+H} yields asymptotic formulae which can be used to analyze
number-theoretical conjectures in the orthogonal
and symplectic cases \ci{BK}, in the same way that
the Toeplitz formulae of Theorem \ref{asTop} were used in the unitary case in \ci{GHK}.
(Note that in \cite{GHK}, the authors only need the version of Theorem \ref{asTop} originally proven
by Widom \cite{W}).

Another example is the probability $P_E(n)$ of a ferromagnetic string of length $n$ in
the antiferromagnetic ground state in the $XY$ spin chain. For a certain range of parameters,
$P_E(n)$ is given by a Toeplitz determinant with 2 $\bt$-singularities such that $|||\bt|||=1$.
Thus, Theorem \ref{BT1} verifies the result on $P_E(n)$, $n\to\infty$, presented in \ci{FA},
which the  authors based on the Basor-Tracy conjecture.

In a similar vein, our results can be used to
justify the asymptotic results for correlators arising in the theory of the impenetrable
Bose gas, that were obtained in \ci{Ovchinnikov} on the basis of the Basor-Tracy conjecture.

\section{Orthogonal polynomials on the unit circle.
Toeplitz and Hankel determinants.}\la{opuc}

Here we present aspects of the theory of orthogonal polynomials
on the unit circle we use in this work. Some of the properties we describe
here are well-known (see, e.g. \ci{Szego}, \ci{Simon}), the others not so.
We also adapt
the theory to complex weights we need in this work, while in the literature
usually only positive weights are considered.

Let $f(z)$ be a complex-valued function integrable over the unit circle, and
let $\phi_k(z)=\chi_k z^k+\cdots$, $\widehat\phi_k(z)=\chi_k z^{k}+\cdots$,
$k=0,1,\dots$ be a system of polynomials in $z$ of degree $k$ with the same for
$\phi_k(z)$ and $\widehat\phi_k(z)$ leading coefficients $\chi_k$.
These polynomials are called orthonormal on the unit circle with weight $f(z)$ if
they satisfy (\ref{or0}).
If $f(z)$ is positive on the unit circle, it is a classical fact that
$D_n(f)\neq 0$ for all $n\ge 1$, and such a system of polynomials exists.
In general, suppose
that all the Toeplitz determinants $D_n$, $n=1,2,\dots$ (\ref{TD}) are nonzero, $D_0\equiv 1$.
Then the polynomials $\phi_k(z)$ and $\widehat\phi_k(z)$ for $k=0,1,\dots$ are given by
the explicit formulae (\ref{ef1}), (\ref{ef2}) for {\it all} $k=1,2,\dots$.
For $k=0$ set
\be\la{k=0}
 \phi_0(z)=\widehat\phi_0(z)=\chi_0=1/\sqrt{D_1}.
\ee
Relations (\ref{or0}) are then equivalent to
\be\la{or}
{1\over 2\pi}\int_0^{2\pi}\phi_k(z)\widehat\phi_m(z^{-1})f(z)d\theta=\de_{km},\qquad
z=e^{i\theta},\qquad k,m=0,1,\dots.
\ee
Thus we constructed
the system of orthogonal polynomials under condition that all the Toeplitz
determinants are nonzero. If we only know that $D_n(f)\neq 0$ for all $n\ge N_0$ with some $N_0>0$,
then we have the existence of $\phi_k(z)=\chi_k z^k+\cdots$, $\widehat\phi_k(z)=\chi_k z^{k}+\cdots$,
$\chi_k\neq 0$, satisfying (\ref{or0}) for $k=N_0,N_0+1,\dots$.

\begin{remark}
From (\ref{ef1},\ref{ef2},\ref{k=0}) we easily conclude:

a) If $f(z)$ is real on the unit circle, we have
$\widehat\phi_n(z^{-1})=
\overline{\phi_n(z)}$, $n=0,1,\dots$, on the unit circle.

b) If $f(e^{i\theta})=f(e^{-i\theta})$, then $\widehat\phi_n(z^{-1})=\phi_n(z^{-1})$.
\end{remark}

\begin{lemma}[Recurrence relations]\la{Lemma1}
Let $D_n(f)\neq 0$, $n\ge 0$.
The orthogonal polynomials satisfy the following relations for $n=0,1,\dots$:
\begin{eqnarray}\la{rr1}
\chi_n z\phi_n(z)=\chi_{n+1}\phi_{n+1}(z)-
\phi_{n+1}(0)z^{n+1}\widehat\phi_{n+1}(z^{-1});\\
\la{rr1p}
\chi_n z^{-1}\widehat\phi_n(z^{-1})=
\chi_{n+1}\widehat\phi_{n+1}(z^{-1})-\widehat\phi_{n+1}(0)z^{-n-1}\phi_{n+1}(z);\\
\la{rr2}
\chi_{n+1}z^{-1}\widehat\phi_n(z^{-1})=
\chi_n\widehat\phi_{n+1}(z^{-1})-\widehat\phi_{n+1}(0)z^{-n}\phi_n(z).
\end{eqnarray}
Moreover,
\be\la{rr3}
\chi_{n+1}^2-\chi_n^2=\phi_{n+1}(0)\widehat\phi_{n+1}(0).
\ee
\end{lemma}

\begin{proof}
To prove (\ref{rr1}) consider the function
\[
g(z)=\chi_n\phi_n(z)-\chi_{n+1}z^{-1}\phi_{n+1}(z)+
\phi_{n+1}(0)z^n \widehat\phi_{n+1}(z^{-1}).
\]
We see that it has zero coefficient at $z^{-1}$ and so $g(z)$ is a polynomial
in $z$ of degree $n$. Therefore we can write
\[
g(z)=\sum_{k=0}^n c_k \phi_k(z),
\]
where $c_k={1\over 2\pi}\int_0^{2\pi}g(z)\widehat\phi_k(z^{-1})f(z)d\theta$.
This integral is easy to calculate using the orthogonality in the form of (\ref{or0})
(for example,
${1\over 2\pi}\int_0^{2\pi}\phi_{n+1}(z)z^{-1}\widehat\phi_k(z^{-1})f(z)d\theta=
(\chi_n/\chi_{n+1})\de_{nk}$), and we obtain that all $c_k=0$. Thus $g(z)\equiv 0$
and (\ref{rr1}) is proved.

Similarly, considering $g_1(z)=\chi_n\widehat\phi_n(z^{-1})-\chi_{n+1}z\widehat\phi_{n+1}(z^{-1})+\widehat\phi_{n+1}(0)z^{-n}\phi_{n+1}(z)$ we show that $g_1(z)\equiv 0$, which proves equation (\ref{rr1p}).

Collecting the coefficients at $z^{n+1}$ in (\ref{rr1}) we obtain (\ref{rr3}).

Finally, multiplying (\ref{rr1}) by $z^{-n-1}\widehat\phi_{n+1}(0)$, and (\ref{rr1p}) by $\chi_{n+1}$,
adding the resulting equations together and using (\ref{rr3}), we obtain
(\ref{rr2}).
\end{proof}

\begin{lemma}[Christoffel-Darboux identity]
Let $D_n(f)\neq 0$, $n\ge 0$.
For any $z$, $a\neq 0$, $n=1,2,\dots$,
\be\la{CD1}
(1-a^{-1}z)\sum_{k=0}^{n-1}\widehat\phi_k(a^{-1})\phi_k(z)=
a^{-n}\phi_n(a) z^n \widehat\phi_n(z^{-1})-\widehat\phi_n(a^{-1})\phi_n(z).
\ee
For any $z\neq 0$, $n=1,2,\dots$,
\be\la{CD2}
\sum_{k=0}^{n-1}\widehat\phi_k(z^{-1})\phi_k(z)=
-n\phi_n(z)\widehat\phi_n(z^{-1}) +z\left(\widehat\phi_n(z^{-1}){d\over dz}\phi_n(z)-
\phi_n(z){d\over dz}\widehat\phi_n(z^{-1})\right).
\ee
\end{lemma}

\begin{proof}
Consider $(1-a^{-1}z)\widehat\phi_k(a^{-1})\phi_k(z)$, for a fixed
$k\ge0$.
Using the recurrence relation (\ref{rr1}) with $n=k$ to express $z\phi_k(z)$ in
terms of $\phi_{k+1}(z)$ and $\widehat\phi_{k+1}(z^{-1})$, and using (\ref{rr2})
with $n=k$ to express $a^{-1}\widehat\phi_k(a^{-1})$, we obtain:
\begin{multline}\nonumber
(1-a^{-1}z)\widehat\phi_k(a^{-1})\phi_k(z)=
\widehat\phi_k(a^{-1})\phi_k(z)-\widehat\phi_{k+1}(a^{-1})\phi_{k+1}(z)\\
+{\phi_{k+1}(0)\over \chi_{k+1}}z^{k+1}\widehat\phi_{k+1}(z^{-1})a^{-k-1}a^{k+1}
\widehat\phi_{k+1}(a^{-1})+
{\widehat\phi_{k+1}(0)\over \chi_{k}}a^{-k}\phi_{k}(a)z^{k+1}z^{-k-1}
\widehat\phi_{k+1}(z)\\
-{\phi_{k+1}(0)\widehat\phi_{k+1}(0)\over \chi_{k}\chi_{k+1}}
a^{-k}\phi_{k}(a)z^{k+1}\widehat\phi_{k+1}(z^{-1}).
\end{multline}
Now expressing in the third summand $a^{k+1}\widehat\phi_{k+1}(a^{-1})$ from
(\ref{rr1}) with $n=k$ and $z=a$, and in the fourth summand
$z^{-k-1}\widehat\phi_{k+1}(z)$ from (\ref{rr1p}), and by using (\ref{rr3}),
we obtain
\begin{multline}\nonumber
(1-a^{-1}z)\widehat\phi_k(a^{-1})\phi_k(z)=
\widehat\phi_k(a^{-1})\phi_k(z)-\widehat\phi_{k+1}(a^{-1})\phi_{k+1}(z)\\
+a^{-k-1}\phi_{k+1}(a)z^{k+1}\widehat\phi_{k+1}(z^{-1})-
a^{-k}\phi_{k}(a)z^{k}\widehat\phi_{k}(z^{-1}).
\end{multline}
Summing this over $k$ from $k=0$ to $n-1$ yields (\ref{CD1}).
Taking the limit $a\to z$ in  (\ref{CD1}) gives  (\ref{CD2}).
\end{proof}

The next lemma allows us to represent the Toeplitz determinant with
symbol $z^\ell f(z)$,
where $\ell$ is any integer, in terms of the one with symbol $f(z)$.

\begin{lemma}\la{Chr}
Let the Toeplitz determinants $D_n(f)$ with symbol $f(z)$ be nonzero
for all $n\ge N_0$ with a fixed $N_0\ge 0$. Let
$\Phi_k(z)=\phi_k(z)/\chi_k$, $\widehat\Phi_k(z)=\widehat\phi_k(z)/\chi_k$, $k=N_0,N_0+1,\dots$
be the system of monic polynomials
orthogonal on the unit circle with the weight $f(z)$.
Fix an integer $\ell>0$. Then if
\[
F_k=\left|
\begin{matrix}
\Phi_k(0)& \Phi_{k+1}(0) & \cdots & \Phi_{k+\ell-1}(0)\cr
{d\over dz}\Phi_k(0) & {d\over dz}\Phi_{k+1}(0)& \cdots & {d\over dz}\Phi_{k+\ell-1}(0)\cr
\vdots & \vdots &  & \vdots\cr
 {d^{\ell-1}\over dz^{\ell-1}}\Phi_k(0) & {d^{\ell-1}\over dz^{\ell-1}}\Phi_{k+1}(0)& \cdots &
{d^{\ell-1}\over dz^{\ell-1}}\Phi_{k+\ell-1}(0)
\end{matrix}
\right|\neq 0, \qquad k=N_0,N_0+1,\dots, n-1,
\]
we have
\be\la{213}
D_n(z^\ell f(z))={(-1)^{\ell n} F_n\over \prod_{j=1}^{\ell-1}j!}
D_n(f(z)),\qquad n\ge N_0.
\ee
In particular, for $\ell=1$, if $\phi_k(0)\neq 0$, $k=N_0,N_0+1,\dots,n-1$, we have
\be\la{DD1}
D_n(z f(z))=(-1)^n{\phi_n(0)\over \chi_n}D_n(f(z)),\qquad n\ge N_0.
\ee

Furthermore, if
\[
\widehat F_k=\left|
\begin{matrix}
\widehat\Phi_k(0)& \widehat\Phi_{k+1}(0) & \cdots &  \widehat\Phi_{k+\ell-1}(0)\cr
{d\over dz} \widehat\Phi_k(0) & {d\over dz} \widehat\Phi_{k+1}(0)& \cdots &
{d\over dz} \widehat\Phi_{k+\ell-1}(0)\cr
\vdots & \vdots &  & \vdots \cr
 {d^{\ell-1}\over dz^{\ell-1}} \widehat\Phi_k(0) & {d^{\ell-1}\over dz^{\ell-1}}\widehat\Phi_{k+1}(0)&
\cdots & {d^{\ell-1}\over dz^{\ell-1}} \widehat\Phi_{k+\ell-1}(0)
\end{matrix}
\right|\neq 0, \qquad k=N_0,N_0+1,\dots, n-1,
\]
we have
\be
D_n(z^{-\ell} f(z))= {(-1)^{\ell n} \widehat F_n\over \prod_{j=1}^{\ell-1}j!}
D_n(f(z)),\qquad n\ge N_0.
\ee
In particular, for $\ell=1$, if $ \widehat\phi_k(0)\neq 0$, $k=N_0,N_0+1,\dots,n-1$, we have
\be\la{DD-1}
D_n(z^{-1} f(z))=(-1)^n{\widehat\phi_n(0)\over \chi_n}D_n(f(z)),\qquad n\ge N_0.
\ee
\end{lemma}

\begin{proof}
We give the proof for $\ell=1$; the generalization is a simple exercise.
Recall that since $D_n\neq 0$, $n=N_0,N_0+1,\dots$,
the polynomials $\phi_n(z)=\chi_n z^n+\dots$, $\chi_n\neq 0$, exist for
$n=N_0,N_0+1,\dots$.

Assume first that $N_0=0$.
Given the polynomials $\phi_k(z)$ related to the weight $f(z)$, we will need the ones
corresponding to the weight $z f(z)$. An analogous construction for polynomials
orthogonal on the real line is known as Christroffel's formula
(see \ci{Szego}, p. 333).
Namely, define $q_k(z)$ by the expression:
\be\la{qdef}
z q_k(z)=\left|
\begin{matrix}
\phi_k(z) & \phi_{k+1}(z)\cr
\phi_k(0) & \phi_{k+1}(0)
\end{matrix}
\right|.
\ee
We see immediately that $q_k(z)$ is a polynomial, and if $\phi_k(0)\neq 0$,
it has degree $k$ with leading coefficient $-\chi_{k+1}\phi_k(0)$.
Moreover, by orthogonality,
\[
\int_0^{2\pi}q_k(z) z^{-j} zf(z)d\theta=0,\qquad j=0,1,\dots,k-1.
\]
For $j=k$,
\[
{1\over 2\pi}\int_0^{2\pi}q_k(z) z^{-k} zf(z)d\theta=
\left|
\begin{matrix}
1/\chi_k & 0\cr
\phi_k(0) & \phi_{k+1}(0)
\end{matrix}
\right|=
{\phi_{k+1}(0)\over \chi_k}.
\]
Therefore, for {\it monic} polynomials $Q_k(z)=q_k(z)/(-\chi_{k+1}\phi_k(0))$,
\[
{1\over 2\pi}\int_0^{2\pi}Q_k(z) z^{-k} zf(z)d\theta=
-{\phi_{k+1}(0)\over\phi_k(0)}{1\over \chi_k\chi_{k+1}}\equiv h_k.
\]
Thus, cf. (\ref{chiD}), the Toeplitz determinant with symbol $z f(z)$, is given by the expression
\be\la{214}
D_n(zf(z))=\prod_{k=0}^{n-1} h_k=
{\phi_n(0)\over \phi_0(0)}{(-1)^n\over \chi_0\chi_1^2\cdots\chi_{n-1}^2\chi_n}=
(-1)^n{\phi_n(0)\over \chi_n}D_n(f(z)),
\ee
which is equation (\ref{DD1}).
The case of $z^{-1}f(z)$, i.e. equation (\ref{DD-1}), is obtained similarly by
considering $\widehat\phi_k(z^{-1})$ instead of $\phi_k(z)$. Namely, we start with
the definition
\[
z^{-1}\wt q_k(z^{-1})=\left|
\begin{matrix}
\widehat\phi_k(z^{-1}) & \widehat\phi_{k+1}(z^{-1})\cr
\widehat\phi_k(0) & \widehat\phi_{k+1}(0)
\end{matrix}
\right|
\]
and proceed as before.

Suppose now that $D_n(f)\neq 0$ for $n\ge 0$, but $F_k$, $\widehat F_k$ are known to be nonzero only
for $k=N_0,N_0+1,\dots, n-1$, with some $N_0>0$.
Consider the polynomials $F_k(z)$, $\widehat F_k(z)$ defined as $F_k$, $\widehat F_k$
with the argument 0 of the orthogonal polynomials replaced with $z$. Obviously, the
set $\Om$ of possible zeros of $F_k(z)$, $\widehat F_k(z)$ for $k=0,1,\dots,N_0-1$ is finite.
We now replace (\ref{qdef}) with
\[
(z-t) q_k(z)=\left|
\begin{matrix}
\phi_k(z) & \phi_{k+1}(z)\cr
\phi_k(t) & \phi_{k+1}(t)
\end{matrix}
\right|,
\]
and choose $t$ so that $\phi_k(t)\neq 0$ for $k=0,1,\dots,N_0-1$.
Instead of (\ref{214}) consider the product
\[
\prod_{k=0}^{N_0-1}
\left[-{\phi_{k+1}(t)\over\phi_k(t)}{1\over \chi_k\chi_{k+1}}\right]
\prod_{k=N_0}^{n-1}
\left[-{\phi_{k+1}(0)\over\phi_k(0)}{1\over \chi_k\chi_{k+1}}\right]
\]
and take the limit $t\to 0$ so that $t$ avoids the set $\Om$. This proves
equation (\ref{DD1}) under the condition $D_n(f)\neq 0$, $n\ge 0$.
We extend the result to a weaker condition $D_n(f)\neq 0$, $n=N_0,N_0+1,\dots$, and thus
complete the proof of (\ref{DD1}), by using the fact that $D_n(f)\neq 0$, $n=1,2,\dots$,
for positive $f$ on the unit circle, and by a simple continuity argument in $\al_j$ and $\bt_j$
(cf. \ci{IK}).
The case of $z^{-1}f(z)$ is dealt with similarly.
\end{proof}

We will now establish a connection between a Hankel determinant with symbol
on a finite interval and a Toeplitz determinant. First we need a theorem
due to Szeg\H o on a relation between polynomials orthogonal on an
interval of the
real axis and those orthogonal on the unit circle. Szeg\H o
considered positive
weights on the unit circle, but his theorem is transferred to the general
case without much change:

\begin{lemma}\la{Pszego}
Let $f(z)$ have the property $f(e^{i\theta})=f(e^{-i\theta})$,
$0\le\theta\le 2\pi$ and let
\[
w(x)={f(e^{i\theta})\over |\sin\theta|},\qquad x=\cos\theta.
\]
Assume that $D_n(f)\neq 0$, $n\ge N_0$, $N_0\ge 0$.
Then the polynomials $p_n(x)=\varkappa_n x^n+\cdots$, $n=N_0,N_0+1,\dots$ exist which are orthonormal w.r.t.
weight $w(x)$ on $[-1,1]$, i.e.,
\[
\int_{-1}^1 p_n(x)x^m w(x)dx=\varkappa_k^{-1}\de_{nm},\qquad m=0,1,\dots,n,\quad n\ge N_0,
\]
and, for $n=N_0,N_0+1,\dots$, there hold the following expressions in terms of the polynomials $\phi_n(z)$
orthogonal w.r.t. $f(z)$ on the unit circle:
\begin{eqnarray}\la{vk}
\varkappa_n=2^n\chi_{2n}\sqrt{1-a_{2n-1}\over 2\pi},\\
\la{PPhi}
P_n(x)={1\over (2z)^n (1-a_{2n-1})}(\Phi_{2n}(z)+\Phi^*_{2n}(z)),\qquad n\ge N_0,
\end{eqnarray}
where
$P_n(z)=p_n(z)/\varkappa_n$, $\Phi_n(z)=\phi_n(z)/\chi_n$,
$\Phi^*_n(z)=z^n\Phi_n(z^{-1})$,
$a_{n-1}=-\Phi_n(0)$, $n\ge N_0$.
\end{lemma}

\begin{proof}
The condition on Toeplitz determinants immediately implies the existence of the polynomials
$\phi_n(z)=\chi_n z^n+\cdots$, $\chi_n\neq 0$,
$n=N_0,N_0+1,\dots$ orthogonal w.r.t. $f(z)$ on the unit circle.
By Remark (b) above, in the present case of $f(e^{i\theta})$ being
an even function of $\theta$, we have $\widehat\phi_n(z^{-1})= \phi_n(z^{-1})$ for all $n\ge N_0$. Now the proof is the same as the argument in the proof of Theorem 11.5 in \ci{Szego}, and we obtain
\be\la{pphi}
p_n(x)={1\over \sqrt{2\pi(1-a_{2n-1})}}
(z^{-n}\phi_{2n}(z)+z^n\phi_{2n}(z^{-1})).
\ee
Note that $1-a_{2n-1}\neq 0$, $n=N_0,N_0+1,\dots$ as follows from
(\ref{rr3}) which in our case can be rewritten in the form
\[
\chi_n^2=\left(1-{\phi_{n+1}^2(0)/\chi_{n+1}^2}\right)\chi_{n+1}^2.
\]
We now easily obtain the statement of the lemma from (\ref{pphi}).
\end{proof}

Note that the recurrence relation (\ref{rr2}) can be easily
rewritten in terms of the monic polynomials
in the form (for $f(e^{i\theta})$ an even function of $\theta$ and with $z$ is replaced by $z^{-1}$):
\be\la{RR1}
\Phi_{n+1}(z)=z\Phi_n(z)-a_n\Phi^*_n(z).
\ee
Replacing here again $z$ by $z^{-1}$ and multiplying both sides by
$z^{n+1}$ we obtain
\be\la{RR11}
\Phi^*_{n+1}(z)=\Phi^*_n(z)-a_n z\Phi_n(z).
\ee

Now we are ready to formulate and prove

\begin{theorem}\la{HT}[Connection between Toeplitz and Hankel determinants]
Let $N_0\ge 0$ and $D_n(f)\neq 0$ for all $n\ge N_0$. Let the weights $f(z)$ and $w(x)$
be related as in Lemma \ref{Pszego}.
Let, moreover,
\[
D_n(w(x))=\det\left(\int_{-1}^1 x^{j+k}w(x)dx\right)_{j,k=0}^{n-1},\qquad
n=N_0,N_0+1,\dots
\]
be the Hankel determinant with symbol $w(x)$ on $[-1,1]$.
Then, with $\Phi_n(z)=\phi_n(z)/\chi_n$, we have
\be\la{HTdet}
 D_n(w(x))^2={\pi^{2n}\over 4^{(n-1)^2}}
(1+\Phi_{2n}(0))^2{D_{2n}(f(z))\over\Phi_{2n}(1)\Phi_{2n}(-1)},\qquad
n=N_0,N_0+1,\dots
\ee
\end{theorem}

\begin{proof}
Assume first $N_0=0$.
Take equation (\ref{PPhi}) with $n=k+1$ and apply the
recurrence relations (\ref{RR1},\ref{RR11}) with $n=2k+1$ to $\Phi_{2k+2}(z)$
and $\Phi^*_{2k+2}(z)$, respectively. We then obtain
\[
P_{k+1}(x)=(2z)^{-k-1}(z\Phi_{2k+1}(z)+\Phi^*_{2k+1}(z)).
\]
Now apply again the relations (\ref{RR1},\ref{RR11}) with $n=2k$ to
$\Phi_{2k+1}(z)$ and $\Phi^*_{2k+1}(z)$ here, respectively.
The result can be written in the form
\[
\Phi^*_{2k}(z)={(2z)^{k+1}\over 1-z a_{2k}}P_{k+1}(x)-z{z-a_{2k}\over
1-z a_{2k}}\Phi_{2k}(z),
\]
where we assume that $z\neq 0$ and $1-z a_{2k}\neq 0$.
On the other hand, from (\ref{PPhi}) with $n=k$
\[
\Phi^*_{2k}(z)=(2z)^k(1-a_{2k-1})P_k(x)-\Phi_{2k}(z)
\]
Equating the r.h.s. of the last two equations, we obtain
\be
(z^2-1)\Phi_{2k}(z)=(2z)^{k+1}P_{k+1}(x)-(2z)^k(1-a_{2k-1})(1-za_{2k})P_k(x).
\ee
Setting here $z=1$ (recall from the proof of Lemma \ref{Pszego} that
$1\pm a_n\neq 0$, $n=0,1,\dots$), we obtain
\[
(1-a_{2k-1})(1-a_{2k})=2{P_{k+1}(1)\over P_k(1)}.
\]
Note that it is a general property of orthogonal polynomials that $P_k(x)$ and $P_{k+1}(x)$
cannot have a zero in common.
Similarly, setting $z=-1$, we have
\[
(1-a_{2k-1})(1+a_{2k})=-2{P_{k+1}(-1)\over P_k(-1)}.
\]
The product of these two equations yields
\[
(1-a_{2k-1})^2(1-a_{2k}^2)=-4{P_{k+1}(1)\over P_k(1)}
{P_{k+1}(-1)\over P_k(-1)}.
\]
By the relation (\ref{rr3}), we can substitute here
\[
1-a_{2k}^2=\left({\chi_{2k}\over\chi_{2k+1}}\right)^2,
\]
which gives
\be\la{aP}
(1-a_{2k-1})^2=-4 \left({\chi_{2k+1}\over\chi_{2k}}\right)^2{P_{k+1}(1)P_{k+1}(-1)
\over P_k(1) P_k(-1)}.
\ee
This equation together with (\ref{vk}) and the well-known expression for
$D_n(w(x))$ in terms of the leading coefficients $\varkappa_k^{-2}$ implies
\be
D_n(w(x))^2=\prod_{k=0}^{n-1}\varkappa_k^{-4}=
{\pi^{2n}\over 4^{n(n-1)}}{(-1)^n\over P_n(1)P_n(-1)}
\prod_{k=0}^{2n-1}\chi_k^{-2}=
{\pi^{2n}\over 4^{n(n-1)}}{(-1)^n\over P_n(1)P_n(-1)}D_{2n}(f(z)).
\ee
Now using (\ref{PPhi}), we obtain
\[
P_n(1)P_n(-1)={(-1)^n\over 4^{n-1} (1-a_{2n-1})^2}\Phi_{2n}(1)\Phi_{2n}(-1),
\]
and thus finish the proof for $N_0=0$. (Note that (\ref{aP}) implies that $\Phi_{2n}(\pm 1)\neq 0$.)
The extension to an arbitrary $N_0>0$ is carried out as in the proof of Lemma \ref{Chr}.
\end{proof}

We will also need a connection between Hankel and Toeplitz+Hankel determinants.
We borrow the idea of the next statement from \ci{Weyl, J2, BR}.

\begin{lemma}\la{HTH}[Connection between Hankel and Toeplitz+Hankel determinants]
 Let $f_j$ be the Fourier coefficient $f_j={1\over 2\pi}\int_0^{2\pi}f(e^{i\theta})e^{-i j\theta}d\theta$. Let $f(e^{i\theta})=f(e^{-i\theta})$.
Then, for $n=1,2,\dots$,
\be\la{TH+}
\det (f_{j-k}+f_{j+k})_{j,k=0}^{n-1}={2^{n^2-2n+2}\over\pi^n} D_n(v(x)),
\ee
where $D_n(v(x))$ is the Hankel determinant with symbol
$v(x)=f(e^{i\theta(x)})/\sqrt{1-x^2}$,
$x=\cos\theta$ on $[-1,1]$.
Furthermore, again in terms of Hankel determinants with symbols on $x\in[-1,1]$,
\begin{eqnarray}
\det (f_{j-k}-f_{j+k+2})_{j,k=0}^{n-1}={2^{n^2}\over\pi^n}
D_n(f(e^{i\theta(x)})\sqrt{1-x^2}),\la{TH-}\\
\det (f_{j-k}+f_{j+k+1})_{j,k=0}^{n-1}={2^{n^2-n}\over\pi^n}
D_n(f(e^{i\theta(x)})\sqrt{1+x\over 1-x}),\la{TH+1}\\
\det (f_{j-k}-f_{j+k+1})_{j,k=0}^{n-1}={2^{n^2-n}\over\pi^n}
D_n(f(e^{i\theta(x)})\sqrt{1-x\over 1+x}).\la{TH-1}
\end{eqnarray}
\end{lemma}

\begin{proof}
Since $f(e^{i\theta})=f(e^{-i\theta})$, note that for $j,k=0,1,\dots$,
\be
{1\over\pi}\int_0^{2\pi}f(e^{i\theta})\cos j\th \cos k\th d\th=
{1\over 2\pi}\int_0^{2\pi}f(e^{i\theta})(e^{-i(j+k)\th}+e^{-i(j-k)\th})d\th=f_{j-k}+f_{j+k}.
\ee
Therefore, using the standard expansion (where only the first coefficient is
needed to be known explicitly) in non-negative powers of the cosine,
\be
\cos k\th =2^{k-1}\cos^k\th+c_{k-2}\cos^{k-2}\th+c_{k-4}\cos^{k-4}\th+\cdots,
\ee
we obtain
\begin{multline}
\det (f_{j-k}+f_{j+k})_{j,k=0}^{n-1}=
\det \left({1\over\pi}\int_0^{2\pi}f(e^{i\theta})\cos j\th\cos k\th d\th
\right)_{j,k=0}^{n-1}\\
=2^{1+2+\cdots+n-2}
\det\left({1\over\pi}\int_0^{2\pi}f(e^{i\theta})\cos j\th\cos^k\th d\th
\right)_{j,k=0}^{n-1}\\
=2^{(n-1)(n-2)}
\det\left({1\over\pi}\int_0^{2\pi}f(e^{i\theta})\cos^j\th\cos^k\th d\th
\right)_{j,k=0}^{n-1}\\
=2^{(n-1)(n-2)}
\left({2\over\pi}\right)^n
\det\left(\int_0^\pi f(e^{i\theta})\cos^j\th\cos^k\th d\th
\right)_{j,k=0}^{n-1}.
\end{multline}
Changing the variable $x=\cos\th$, $d\th=-dx/\sqrt{1-x^2}$, we immediately obtain
\be
\det (f_{j-k}+f_{j+k})_{j,k=0}^{n-1}=
{2^{n^2-2n+2}\over\pi^n}
\det\left(\int_{-1}^1 v(x) x^{j+k}dx\right),\qquad v(x)=
{f(e^{i\theta(x)})\over\sqrt{1-x^2}},
\ee
which is (\ref{TH+}).

Similarly, using the observations
\begin{eqnarray}
{1\over\pi}\int_0^{2\pi}f(e^{i\theta})\sin (j+1)\th \sin (k+1)\th d\th=
f_{j-k}-f_{j+k+2},\\
{1\over\pi}\int_0^{2\pi}f(e^{i\theta})\cos (j+1/2)\th \cos (k+1/2)\th d\th=
f_{j-k}+f_{j+k+1},\\
{1\over\pi}\int_0^{2\pi}f(e^{i\theta})\sin (j+1/2)\th \sin (k+1/2)\th d\th=
f_{j-k}-f_{j+k+1},
\end{eqnarray}
and the expansions in non-negative powers of the cosine of the quantities
\[
{\sin (k+1)\th \over \sin\th},\qquad
{\cos (k+1/2)\th \over \cos{\th\over 2}},\qquad
{\sin (k+1/2)\th \over \sin{\th\over 2}},
\]
we obtain (\ref{TH-}), (\ref{TH+1}), and (\ref{TH-1}).
\end{proof}

Finally, we list some properties of Barnes' $G$-function (see \ci{Barnes,WW})
we need below. The $G$-function is an entire function defined, e.g.,
by the product:
\be
G(z+1)=(2\pi)^{z/2}e^{-(z+1)z/2-\ga_E z^2/2}
\prod_{k=1}^\infty \left(1+{z\over k}\right)^k e^{-z+z^2/(2k)},\qquad z\in\bbc
\ee
where $\ga_E$ is Euler's constant.
$G(z)$ satisfies the recurrence relation:
\be
G(z+1)=\Ga(z)G(z),\qquad G(1)=1,
\ee
where $\Ga(z)$ is Euler's $G$-function.
The following representation is useful
\be
\int_0^z\ln \Gamma(x+1)dx=
{z\over 2}\ln 2\pi -{z(z+1)\over2}+z\ln\Gamma(z+1)-\ln G(z+1).
\ee
There holds the identity:
\be\la{G12}
2\ln G(1/2)=
(1/12)\ln2-\ln\sqrt\pi+3\zeta'(-1),
\ee
where $\zeta'(x)$ is the derivative of Riemann's $\zeta$-function.
We will also need a doubling formula given by
\be\la{Gdouble}
G(2z)\pi^z G(1/2)^2
=G(z)^2 G(z+1/2)^2\Ga(z) 2^{(2z-1)(z-1)}.
\ee

\section{Riemann-Hilbert problem}\la{RHsection}
In this section we formulate a Riemann-Hilbert problem (RHP)
for the polynomials
$\phi_k(z)$, $\widehat\phi_k(z)$. We use this RHP in section 5 to find
asymptotics of the polynomials.

Let the weight $f(z)$ be given on the unit circle
(which, oriented in the positive direction, we denote $C$)
by (\ref{fFH}).
Suppose that the system of orthonormal polynomials satisfying (\ref{or0}) exists.
Consider the following $2\times 2$ matrix valued function $Y^{(k)}(z)\equiv Y(z)$:
\begin{equation} \label{RHM}
    Y^{(k)}(z) =
    \begin{pmatrix}
\chi_k^{-1}\phi_k(z) &
\chi_k^{-1}\int_{C}{\phi_k(\xi)\over \xi-z}
{f(\xi)d\xi \over 2\pi i \xi^k} \cr
-\chi_{k-1}z^{k-1}\widehat\phi_{k-1}(z^{-1}) &
-\chi_{k-1}\int_{C}{\widehat\phi_{k-1}(\xi^{-1})\over \xi-z}
{f(\xi)d\xi \over 2\pi i \xi}
    \end{pmatrix}.
\end{equation}

It is easy to verify that
$Y(z)$ solves the following Riemann-Hilbert problem:
\begin{enumerate}
    \item[(a)]
        $Y(z)$ is  analytic for $z\in\bbc \setminus C$.
    \item[(b)]
Let $z\in C\setminus\cup_{j=0}^m z_j$.
$Y$ has continuous boundary values
$Y_{+}(z)$ as $z$ approaches the unit circle from
the inside, and $Y_{-}(z)$, from the outside,
related by the jump condition
\begin{equation}\label{RHPYb}
            Y_+(z) = Y_-(z)
            \begin{pmatrix}
                1 & z^{-k}f(z) \cr
                0 & 1
             \end{pmatrix},
            \qquad\mbox{$z\in C\setminus\cup_{j=0}^m z_j$.}
        \end{equation}
    \item[(c)]
        $Y(z)$ has the following asymptotic behavior at infinity:
        \begin{equation} \label{RHPYc}
            Y(z) = \left(I+ O \left( \frac{1}{z} \right)\right)
            \begin{pmatrix}
                z^{k} & 0 \cr
                0 & z^{-k}\end{pmatrix}, \qquad \mbox{as $z\to\infty$.}
        \end{equation}
\item [(d)]
As $z\to z_j$, $j=0,1,\dots,m$, $z\in\bbc\setminus C$,
\be
Y(z)=
\begin{pmatrix} O(1) & O(1)+O(|z- z_j|^{2\al_j})\cr
O(1) & O(1)+O(|z- z_j|^{2\al_j})
\end{pmatrix},
\qquad \mbox{if $\al_j\neq 0$},
\ee
and
\be
Y(z)=
\begin{pmatrix} O(1) & O(\ln|z- z_j|)\cr
O(1) & O(\ln|z-z_j|)
\end{pmatrix},
\qquad \mbox{if $\al_j=0$, $\bt_j\neq 0$}.
\ee
\end{enumerate}
(Here and below $O(a)$ stands for $O(|a|)$. If $\al_0=\bt_0=0$ then $Y(z)$ is bounded at $z=1$.)

A general fact that orthogonal polynomials can be so represented as a solution
of a Riemann-Hilbert problem was noticed in \cite{FIK} (for polynomials
on the line) and extended for polynomials on the circle in \cite{BDJ}.
This fact is important because it turns out that the RHP can be efficiently analyzed
for large $k$ by a steepest-descent-type method found in \ci{DZ} and developed
further in many subsequent works. Thus, we first
find the solution to the problem
(a)--(d) for large $k$ (applying this method)
and then interpret it as the asymptotics of
the orthogonal polynomials by (\ref{RHM}).

The solution to the RHP (a)--(d) is unique. Note first
that $\det Y(z)=1$. Indeed, from the conditions on $Y(z)$,
$\det Y(z)$ is analytic across the unit circle, has all singularities
removable, and tends to $1$ as $z\to\infty$.
It is then identically $1$ by Liouville's theorem. Now
if there is another solution $\wt Y(z)$, we easily obtain by
Liouville's theorem that $Y(z) \wt Y(z)^{-1}\equiv 1$.

%%%%%%%%%%%%%%%%% RHPa %%%%%%%%%%%%

\section{Asymptotic analysis of the Riemann-Hilbert problem}\la{RHa}

In this section we construct an asymptotic solution to
the Riemann-Hilbert problem (a) -- (d) of Section \ref{RHsection}
for large $k=n$ by the steepest descent method.
All the steps of the analysis are standard apart from construction of the
local parametrix near the points $z_j$. We always assume that $f(z)$ is given
by (\ref{fFH}). In all Section \ref{RHa} we also assume for simplicity that $z_0=1$
is a singularity. However, the results trivially extend to the case $\al_0=\bt_0=0$.
In Section \ref{RHa} we further assume that $V(z)$ is analytic in a
neighborhood of the unit circle.

The first step is the following transformation, which normalizes the problem
at infinity:
\be\la{TY}
T(z)=Y(z)
\begin{cases}
z^{-n\si_3},& |z|>1\cr
I,& |z|<1,
\end{cases}
\ee
where $\si_3=\begin{pmatrix} 1 & 0 \cr 0 & -1 \end{pmatrix}$.
From the RHP for $Y(z)$, we obtain the following
problem for $T(z)$:

\begin{enumerate}
    \item[(a)]
        $T(z)$ is  analytic for $z\in\bbc \setminus C$.
    \item[(b)]
The boundary values of $T(z)$ are related by the jump condition
\begin{equation}
 T_+(z) = T_-(z)
            \begin{pmatrix}
                z^n & f(z) \cr
                0 & z^{-n}\end{pmatrix},\qquad z\in C\setminus\cup_{j=0}^m z_j,
\end{equation}
    \item[(c)]
$T(z)=I+O(1/z)$ as $z\to\infty$,
\end{enumerate}
and the condition (d) remains unchanged.

Now split
the contour as shown in
Figure 1.
\begin{figure}
\centerline{\psfig{file=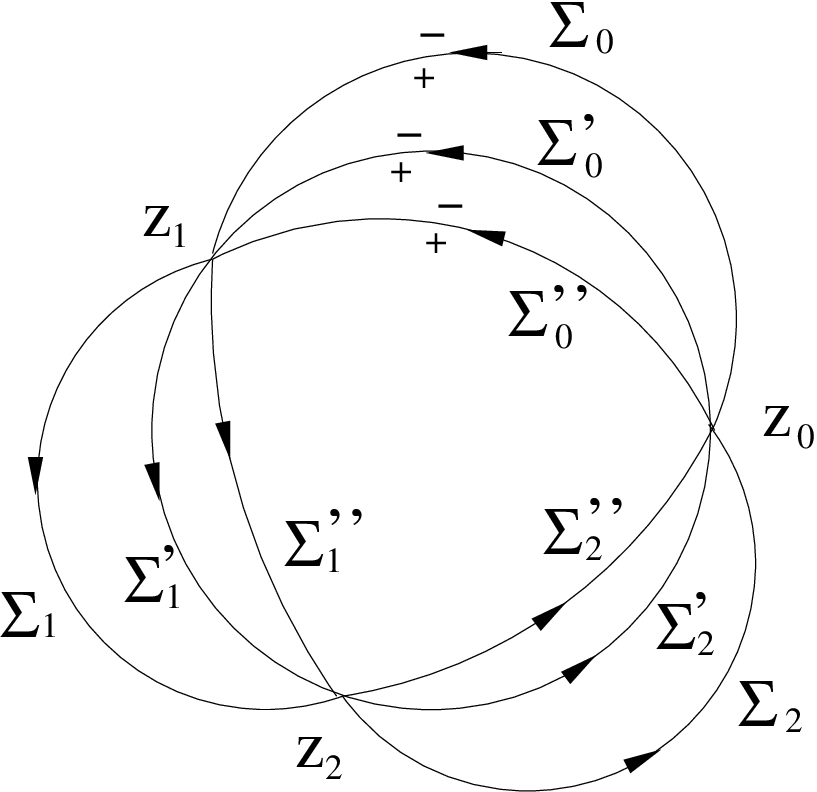,width=3.0in,angle=0}}
\vspace{0cm}
\caption{
Contour for the $S$-Riemann-Hilbert problem ($m=2$).}
\label{fig1}
\end{figure}
Define a new transformation as follows:
\be\la{defS}
S(z)=
\begin{cases}T(z),& \mbox{for $z$ outside the lenses},\cr
T(z)\begin{pmatrix}1 & 0\cr f(z)^{-1}z^{-n}& 1\end{pmatrix}, &
\mbox{for $|z|>1$ and inside the lenses},\cr
 T(z)\begin{pmatrix}1 & 0\cr -f(z)^{-1}z^n& 1\end{pmatrix},&
\mbox{for $|z|<1$ and inside the lenses}.
\end{cases}
\ee
Here $f(z)$ is the analytic continuation of $f(z)$ off $C$ into the inside
of the lenses as discussed in Section \ref{RHparam} below (see (\ref{abszzj}), (\ref{abszzj2})).

Then the Riemann-Hilbert problem for $S(z)$ is the following:
\begin{enumerate}
    \item[(a)]
        $S(z)$ is  analytic for $z\in\bbc \setminus\Sigma$, where
$\Si=\cup_{j=0}^m(\Si_j\cup\Si'_j\cup\Si''_j)$.
    \item[(b)]
The boundary values of $S(z)$ are related by the jump condition
\[
   S_+(z) = S_-(z)
            \begin{pmatrix}
                  1 & 0\cr
                  f(z)^{-1}z^{\mp n } & 1\end{pmatrix},
            \qquad\mbox{$z\in\cup_{j=0}^{m}(\Si_j\cup\Si''_j)$},
\]
where the minus sign in the exponent is on
$\Si_j$, and plus, on $\Si''_j$,
\[
            S_+(z) = S_-(z)
              \begin{pmatrix}
                  0 & f(z)\cr
                  -f(z)^{-1} & 0\end{pmatrix},
            \qquad\mbox{$z\in\cup_{j=0}^{m}\Si'_j$.}
\]
    \item[(c)]
$S(z)=I+O(1/z)$ as $z\to\infty$,
\item [(d)]
As $z\to z_j$, $j=0,\dots,m$, $z\in\bbc\setminus C$ outside the lenses,
\be\la{Sd1}
S(z)=
\begin{pmatrix} O(1) & O(1)+O(|z- z_j|^{2\al_j})\cr
O(1) & O(1)+O(|z- z_j|^{2\al_j})
\end{pmatrix}
\ee
if $\al_j\neq 0$, and
\be\la{Sd2}
S(z)=
\begin{pmatrix} O(1) & O(\ln|z- z_j|)\cr
O(1) & O(\ln|z-z_j|)
\end{pmatrix}
\ee
if $\al_j=0$, $\bt_j\neq 0$.
The behavior of $S(z)$ for $z\to z_j$ in other sectors is obtained from these expressions by application of the appropriate jump conditions.
\end{enumerate}

Let us encircle each of the points
$z_{j}$ by a sufficiently small disc,
\be\la{U}
U_{z_{j}} = \left\{z:|z-z_{j}| < \ep\right\},
\ee
We see that, outside the neighborhoods $U_{z_j}$,
the jump matrix on $\Sigma_j$, $\Sigma''_j$ $j=0,\dots,m$ is
uniformly exponentially close to
the identity.
We will now construct the parametrices in
$\bbc\setminus(\cup_{j=0}^m U_{z_j})$ and $U_{z_j}$.
We match them on the boundaries $\partial U_{z_j}$,
which yields the desired asymptotics.

\subsection{Parametrix outside the points $z_j$}

We expect the following problem for the parametrix $N$ in
$\bbc\setminus\cup_{j=0}^m U_{z_j}$:

\begin{enumerate}
    \item[(a)]
        $N(z)$ is  analytic for $z\in\bbc\setminus C$,
    \item[(b)]
with the jump on $C$
\[
N_{+}(z) = N_{-}(z)
            \begin{pmatrix}0& f(z)\cr
              -f(z)^{-1}&0\end{pmatrix},
\qquad\mbox{$z \in C\setminus\cup_{j=0}^m z_j$},
\]
\item[(c)]
and the following behavior at infinity
\[
N(z) = I+ O \left( \frac{1}{z} \right),
     \qquad \mbox{as $z\to\infty$.}
\]
\end{enumerate}

One can easily check directly that the solution to this RHP
is given by the formula
\be\la{Ndef}
N(z)=
\begin{cases}
\mathcal{D}(z)^{\si_3},& |z|>1\cr
\mathcal{D}(z)^{\si_3}
\bm
0&1\cr -1&0
\em,
& |z|<1
\end{cases},
\ee
where the Szeg\H o function
\be
\mathcal{D}(z)=\exp {1\over 2\pi i}\int_C
{\ln f(s)\over s-z} ds,
\ee
is analytic away from the unit circle with boundary values
satisfying $\mathcal{D}_+(z)=\mathcal{D}_-(z)f(z)$, $z\in C\setminus\cup_{j=0}^m z_j$.

In what follows, we will need a more explicit formula for $\mathcal{D}(z)$.
Calculation of the integral (with the help of (\ref{abszzj}) below) gives:
\be\la{Dl1}
\mathcal{D}(z)=\exp\left[ {1\over 2\pi i}\int_C {V(s)\over s-z} ds\right]
\prod_{k=0}^m\left({z-z_k\over z_k e^{i\pi}}\right)^{\al_k+\bt_k}
=e^{V_0}b_+(z)\prod_{k=0}^m\left({z-z_k\over z_k e^{i\pi}}\right)^{\al_k+\bt_k}
,\qquad
|z|<1.
\ee
and
\be\la{Dg1}
\mathcal{D}(z)=\exp\left[ {1\over 2\pi i}\int_C {V(s)\over s-z} ds\right]
\prod_{k=0}^m\left({z-z_k\over z}\right)^{-\al_k+\bt_k}
=b_-(z)^{-1}\prod_{k=0}^m\left({z-z_k\over z}\right)^{-\al_k+\bt_k}
,\qquad
|z|>1,
\ee
where $V_0$, $b_\pm(z)$ are defined in (\ref{WienH}).
Note that the branch of $(z-z_k)^{\pm\al_k+\bt_k}$ in (\ref{Dl1}), (\ref{Dg1})
is taken as discussed
following equation (\ref{abszzj}) below. In (\ref{Dg1}) for any $k$, the cut
of the root $z^{-\al_k+\bt_k}$ is the line $\th=\th_k$ from
$z=0$ to infinity, and $\th_k<\arg z<2\pi+\th_k$.

%%%%%%%%%%%%% z_j %%%%%%%%%%%%%%%%%%%%%%%%%%%%%%%%%%%%%%%%

\subsection{Parametrix at $z_j$. }\la{RHparam}

Let us now construct the parametrix $P_{z_j}(z)$ in $U_{z_j}$.
The construction is the same for all $j=0,1,\dots$.
We look for an analytic matrix-valued function in a
neighborhood of $U_{z_j}$
which satisfies the same jump conditions as $S(z)$ on
$\Sigma\cap U_{z_j}$, the same conditions (\ref{Sd1},\ref{Sd2})
as $z\to z_j$, and, instead of a condition at infinity,
satisfies the matching condition
\be\la{match}
P_{z_j}(z)N^{-1}(z)=I+o(1)
\ee
uniformly on the boundary $\partial U_{z_j}$ as $n\to\infty$.

First, set
\be\la{zeta}
\zeta=n\ln{z\over z_j},
\ee
where $\ln x>0$ for $x>1$, and has a cut on the negative half of the real axis.
Under this transformation
the neighborhood $U_{z_j}$ is mapped into a neighborhood of zero in the
$\zeta$-plane. Note that $\zeta(z)$ is analytic, one-to-one,
and it takes an arc of the unit circle to an interval of the imaginary axis.
Let us now choose the exact form of the cuts $\Sigma$ in $U_{z_j}$ so that
their images under the mapping $\zeta(z)$ are straight lines (Figure 2).
We add one more jump contour to $\Sigma$ in $U_{z_j}$
which is the pre-image of the real line $\Ga_3$ and $\Ga_7$
in the $\zeta$-plane.
This will be needed below because of the non-analyticity of the function
$|z-z_j|^{\al_j}$.
Note that we can construct two different analytic continuations
of this function off the unit circle to the pre-images of the upper
and lower half $\ze$-plane, respectively.
Namely, write
for $z$ on the unit circle,
\be\la{abszzj}
h_{\al_j}(z)=|z-z_j|^{\al_j}=(z-z_j)^{\al_j/2}(z^{-1}-z_j^{-1})^{\al_j/2}=
\frac{(z-z_j)^{\al_j}}{(z z_j e^{i\ell_j})^{\al_j/2}},\qquad z=e^{i\th},
\ee
where $\ell_j$ is found from the condition that the argument of the above function
is zero on the unit circle.
Let us fix the cut of $(z-z_j)^{\al_j}$ going along the line $\th=\th_j$ from $z_j$
to infinity. Fix the branch by the condition that on the line going from $z_j$ to
the right parallel to the real axis, $\arg (z-z_j)=2\pi $.
For $z^{\al_j/2}$ in the denominator, $0<\arg z<2\pi$ (the same convention for roots of $z$
is adopted in (\ref{Fj},\ref{F0}) below).
Then, a simple consideration of triangles shows that
\be\la{abszzj2}
\ell_j=\begin{cases}
3\pi,& 0<\th<\th_j\cr
\pi, & \th_j<\th<2\pi
\end{cases}.
\ee
Thus (\ref{abszzj}) is continued analytically to neighborhoods of the arcs $0<\th<\th_j$, and
$\th_j<\th< 2\pi$. In $U_{z_j}$, we extend these neighborhoods to the pre-images
of the lower and upper half $\ze$-plane (intersected with $\zeta(U_{z_j})$),
respectively. The cut of $h_{\al_j}$ is along
the contours $\Ga_3$ and $\Ga_7$ in the $\ze$-plane.

For $z\to z_j$, $\zeta= n(z-z_j)/z_j+O((z-z_j)^2)$.
We have $0<\arg\zeta<2\pi$, which follows from the choice of $\arg(z-z_j)$ in (\ref{abszzj}).

Denote by Roman numerals the sectors between the cuts in Figure 2.
We now introduce the following auxiliary function. First, for $j\neq 0$,
\begin{multline}\la{Fj}
F_j(z)
=e^{{V(z)\over 2}} \prod_{k=0}^m
\left(\frac{z}{z_k}\right)^{\bt_k/2}
\prod_{k\neq j} h_{\al_k}(z) g_{\bt_k}(z)^{1/2}\\
\times h_{\al_j}(z)
\begin{cases}
e^{-i\pi\al_j},& \zeta\in I,II,V,VI\cr
e^{i\pi\al_j},& \zeta\in III,IV,VII,VIII
\end{cases},
\quad z\in U_{z_j},\quad j\neq 0.
\end{multline}
The functions $g_{\bt_k}(z)$ are defined in (\ref{g}).
The case of $U_{z_0}$ is slightly different because of the branch cut of
$z^{\bt_k}$ and $z^{\al_k}$ going along
the positive real half-line. Define a step function
\be
\widehat g_{\bt_0}(z)=
\begin{cases}
e^{-i\pi\bt_0},& \arg z>0\cr
e^{i\pi\bt_0},& \arg z<2\pi
\end{cases},\qquad z\in U_{z_0},
\ee
and define
\begin{multline}\la{F0}
F_0(z)
=e^{{V(z)\over 2}}  \prod_{k=0}^m
\left(\frac{z}{z_k}\right)^{\bt_k/2}
\prod_{k\neq 0} h_{\al_k}(z) g_{\bt_k}(z)^{1/2}\\
\times h_{\al_0}(z)
\begin{cases}
e^{-i\pi\al_0},& \zeta\in I,II\cr
e^{i\pi(\al_0-\bt_0)},& \zeta\in III,IV\cr
e^{-i\pi(\al_0+\bt_0)},& \zeta\in V,VI\cr
e^{i\pi\al_0},& \zeta\in VII,VIII
\end{cases},
\quad z\in U_{z_0}.
\end{multline}

It is easy to verify that $F_j(z)$, $j=0,1,\dots$
is analytic in the intersection of
each quarter $\zeta$-plane with $\zeta(U_{z_j})$
and has the following jumps:
\begin{eqnarray}
 F_{j,+}(z)&=&F_{j,-}(z)e^{-2\pi i\al_j}\qquad \ze\in \Ga_1;\\
 F_{j,+}(z)&=&F_{j,-}(z)e^{2\pi i\al_j}\qquad \ze\in \Ga_5;\\
 F_{j,+}(z)&=&F_{j,-}(z)e^{\pi i\al_j}\qquad \ze\in \Ga_3\cup\Ga_7.
\end{eqnarray}

Comparing (\ref{fFH}) and (\ref{Fj}), and using the analytic continuation
(see (\ref{abszzj})) for $f(z)$ off the arcs between the singularities,
we obtain the following relations between $f(z)$ and $F_j(z)$:
\begin{eqnarray}
 F_j(z)^2 &=& f(z)e^{-2\pi i\al_j}g_{\bt_j}^{-1}(z)\qquad \ze\in I,II,V,VI;\\
 F_j(z)^2 &=& f(z)e^{2\pi i\al_j}g_{\bt_j}^{-1}(z)\qquad \ze\in III,IV,VII,VIII.
\end{eqnarray}
for $j\neq 0$. If $j=0$ the same relations hold with the
functions $g_{\bt_0}^{-1}(z)$ replaced by $\widehat g_{\bt_0}^{-1}(z)$.

We look for $P_{z_j}(z)$ in the form
\be\la{Plb0}
P_{z_j}(z)=E(z)P^{(1)}(z)F_j(z)^{-\si_3}z^{\pm n\si_3/2},
\ee
where plus sign is taken for $|z|<1$
(this corresponds to $\ze\in I,II,III,IV$), and minus,
for $|z|>1$ ($\ze\in V,VI,VII,VIII$).
The matrix $E(z)$
is analytic and invertible in the neighborhood of $U_{z_j}$,
and therefore does not affect the jump and analyticity conditions.
It is chosen so that the matching condition is satisfied.

It is easy to verify (recall that $P_{z_j}(z)$ has the same jumps as
$S(z)$) that $P^{(1)}(z)$ satisfies jump
conditions with {\it constant} jump matrices. Set
\be\la{P1}
P^{(1)}(z)=\Psi_j(\zeta).
\ee
Then $\Psi_j(\zeta)$ satisfies a RHP on the contour given in Figure 2:

\begin{figure}
\centerline{\psfig{file=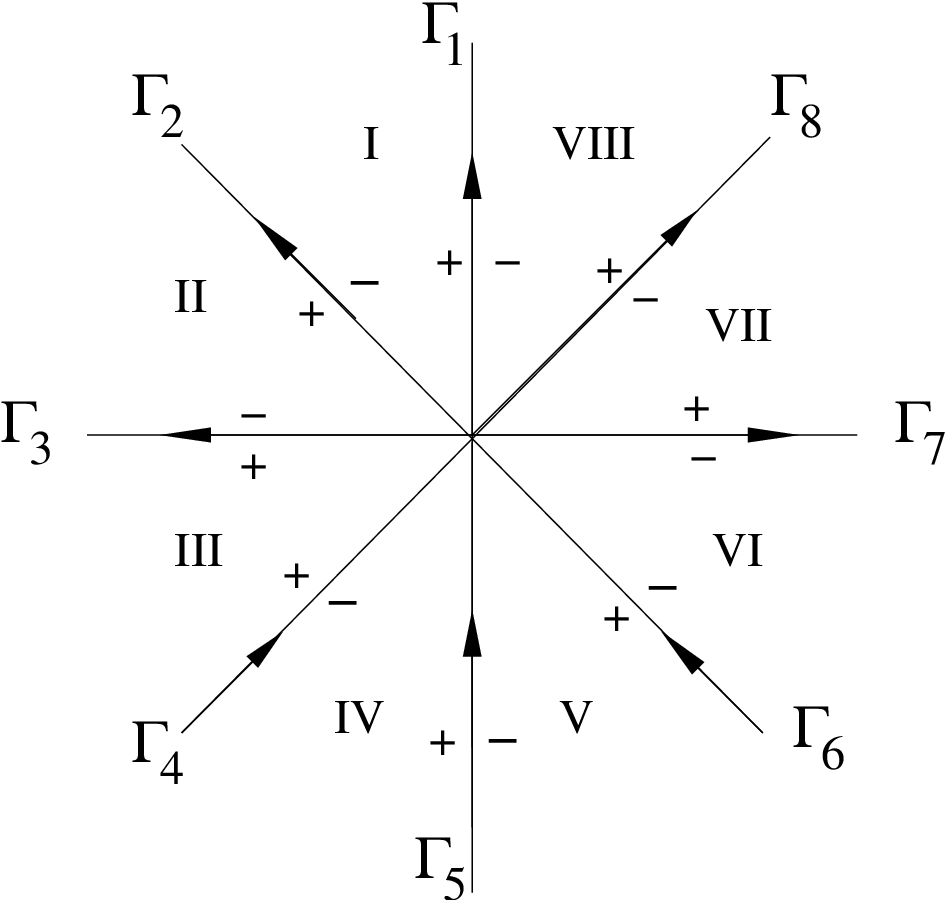,width=3.0in,angle=0}}
\vspace{0cm}
\caption{
The auxiliary contour for the parametrix at $z_j$.}
\label{fig2}
\end{figure}

\begin{enumerate}
\item[(a)]
    $\Psi_{j}$ is analytic for $\ze\in\bbc\setminus\cup_{j=1}^8\Ga_j$.
\item[(b)]
    $\Psi_{j}$ satisfies the following jump conditions:
    \begin{eqnarray}
       \Psi_{j,+}(\zeta)
       &=&
            \Psi_{j,-}(\zeta)
            \begin{pmatrix}
                0 & e^{-i\pi\bt_j} \cr
                -e^{i\pi\bt_j} & 0
            \end{pmatrix},
            \qquad \mbox{for $\zeta \in \Ga_1$,}\label{JumpPsi1}
           \\
        \Psi_{j,+}(\zeta)
        &=&
            \Psi_{j,-}(\zeta)
            \begin{pmatrix}
                0 & e^{i\pi\bt_j} \cr
                -e^{-i\pi\bt_j} & 0
            \end{pmatrix},
            \qquad \mbox{for $\zeta \in \Ga_5$,}\\
        \Psi_{j,+}(\zeta)
        &=&
            \Psi_{j,-}(\zeta)e^{i\pi\al_j\si_3},
              \qquad \mbox{for $\zeta \in \Ga_3\cup\Ga_7$,}\\
        \Psi_{j,+}(\zeta)
        &=&
            \Psi_{j,-}(\zeta)
            \begin{pmatrix}
                1 & 0 \cr
                e^{\pm i\pi(\bt_j-2\al_j)}  & 1
            \end{pmatrix},
\end{eqnarray}
\[
    \mbox{for $\zeta \in \Ga_2$ with plus sign in the exponent,
          for $\zeta \in\Ga_4$, with minus sign,}
\]
\be
        \Psi_{j,+}(\zeta)
        =
            \Psi_{j,-}(\zeta)
            \begin{pmatrix}
                1 & 0 \cr
                e^{\pm i\pi(\bt_j+2\al_j)}  & 1
            \end{pmatrix},\la{JumpPsi8}
\ee
\[
\mbox{for $\zeta \in \Ga_8$ with plus sign in the exponent,
          for $\zeta \in\Ga_6$, with minus sign.}
\]
\item [(c)]
As $\ze\to 0$, $\ze\in\bbc\setminus\cup_{j=1}^8\Ga_j$ outside the lenses,
\be\la{Psic1}
\Psi_j(z)=
\begin{pmatrix} O(\ze^{\al_j}) & O(\ze^{\al_j})+O(\ze^{-\al_j})\cr
 O(\ze^{\al_j}) & O(\ze^{\al_j})+O(\ze^{-\al_j})
\end{pmatrix}
\ee
if $\al_j\neq 0$, and
\be\la{Psic2}
\Psi_j(z)=
\begin{pmatrix} O(1) & O(\ln|\ze|)\cr
O(1) & O(\ln|\ze|)
\end{pmatrix}
\ee
if $\al_j=0$, $\bt_j\neq 0$.
The behavior of $\Psi_j(z)$ for $\ze\to 0$ in other sectors is obtained from these expressions by application of the appropriate jump conditions.
\end{enumerate}

We will solve this problem explicitly in terms of the confluent
hypergeometric function, $\psi(a,c;z)$ with the
parameters $a$, $c$ determined by $\al_j$, $\bt_j$.
A standard theory of the confluent
hypergeometric function is presented, e.g., in the appendix of \ci{IK}.

The following statement holds.

\begin{proposition}\la{Param}
Let $\al_j\pm\bt_j\neq -1,-2,\dots$ for all $j$. Then
a solution to the above RHP (a)--(c) for $\Psi_j(\ze)$, $0<\arg\ze<2\pi$,
is given by the following function in the sector I:
\begin{multline}\label{PsiConfl1}
\Psi_j(\zeta)=\Psi^{(I)}_j(\ze)=\left(\begin{matrix}
\ze^{\al_j}\psi(\al_j+\bt_j,1+2\al_j,\ze)e^{i\pi(2\bt_j+\al_j)}e^{-\ze/2} \cr
-\ze^{-\al_j}
\psi(1-\al_j+\bt_j,1-2\al_j,\ze)e^{i\pi(\bt_j-3\al_j)}e^{-\ze/2}
{\Gamma(1+\al_j+\bt_j)\over\Gamma(\al_j-\bt_j)}
\end{matrix}
\right.\\
\left.
\begin{matrix}
-\ze^{\al_j}
\psi(1+\al_j-\bt_j,1+2\al_j,e^{-i\pi}\ze)e^{i\pi(\bt_j+\al_j)}e^{\ze/2}
{\Gamma(1+\al_j-\bt_j)\over\Gamma(\al_j+\bt_j)}
\cr
\ze^{-\al_j}
\psi(-\al_j-\bt_j,1-2\al_j,e^{-i\pi}\ze)e^{-i\pi\al_j}
e^{\ze/2}\end{matrix}\right),
\end{multline}
where $\psi(a,b,x)$ is the confluent hypergeometric function of the second kind, and
$\Gamma(x)$ is Euler's $\Gamma$-function.
The solution in the other sectors is given by successive application
of the jump conditions (\ref{JumpPsi1}--\ref{JumpPsi8}) to (\ref{PsiConfl1}).
\end{proposition}

\begin{remark}
The functions $\ze^{\pm\al_j}$,  $\psi(a,b,\ze)$, and $\psi(a,b,e^{-i\pi}\ze)$
are defined on the universal covering of the
punctured plane $\zeta \in {\mathbb C}\setminus \{0\}$. Recall that
the branches are fixed by the condition $0<\arg\ze<2\pi$.
\end{remark}

\begin{proof}
The condition $(c)$ is verified in the sector $I$ by applying
to (\ref{PsiConfl1})
the standard expansion of the confluent hypergeometric function at zero
(see, e.g., \ci{BE}), namely,
\begin{multline}
\psi(a,c,x)={\Ga(1-c)\over \Ga(1+a-c)}\left(1+O(x)\right)+
{\Ga(c-1)\over \Ga(a)}x^{1-c}\left(1+O(x)\right),\\
x\to 0,\qquad c\notin\bbz,
\end{multline}
or, to cover also the integer values of $c$:
\begin{multline}\la{psi01}
\psi(a,c,x)=
\begin{cases}
{\Ga(c-1)\over \Ga(a)}x^{1-c}\left(1+O(x\ln x)\right)+O(1),& \Re c>1\cr
{\Ga(1-c)\over \Ga(1+a-c)}\left(1+O(x)\right)+
{\Ga(c-1)\over \Ga(a)}x^{1-c}\left(1+O(x)\right),& \Re c=1, c\neq 1\cr
-{1 \over \Gamma(a)}
\left(\ln x +{\Gamma'(a)\over \Gamma(a)}-2\gamma_E\right)+O(x\ln x),& c=1\cr
{\Ga(1-c)\over \Ga(1+a-c)}\left(1+O(x\ln x)+O(x^{1-c})\right),& \Re c<1
\end{cases},\qquad x\to 0,
\end{multline}
%and
%\be\la{psi02}
%\psi(a,1,x)= -{1 \over \Gamma(a)}
%\left(\ln x +{\Gamma'(a)\over \Gamma(a)}-2C_\Gamma\right)+O(x\ln x),
%\qquad x\to 0,\qquad a \notin\{0,-1,-2,\dots\},
%\ee
where $\gamma_E=0.5772\dots$ is Euler's constant.

We verify the condition $(c)$ similarly in the other sectors.

To verify $(b)$,
reduce the contour of Figure \ref{fig2} to the real line,
oriented from right to left, by
extending the sectors I and IV and collapsing the jump conditions.
We then obtain the following reduced RHP:
\begin{multline}\label{newjump}
\Psi^{(IV)}_{j,+}(\ze)=\Psi^{(I)}_{j,-}(\ze) J_2 J_3 J_4^{-1}=
\Psi^{(I)}_{j,-}(\ze)\bm
                e^{i\pi\al_j} & 0 \cr
                2i\sin(\pi(\bt_j-\al_j))& e^{-i\pi\al_j}\em, \qquad \ze<0;\\
\Psi^{(IV)}_{j,+}(\ze)=\Psi^{(I)}_{j,-}(\ze) J_1^{-1} J_8^{-1} J_7^{-1}J_6 J_5 =
\Psi^{(I)}_{j,-}(\ze)\bm
                e^{-i\pi(2\bt_j-\al_j)} & 2i\sin(\pi(\al_j+\bt_j))\cr
                 0             & e^{i\pi(2\bt_j-\al_j)}
                    \em, \qquad \ze>0,
\end{multline}
where the jump matrices $J_k$ correspond to jumps on the contours $\Ga_k$,
$k=1,\dots, 8$ as defined in (\ref{JumpPsi1}--\ref{JumpPsi8}).

The confluent hypergeometric function possesses
the following transformation property on the universal covering of the
punctured plane:
\be\label{prop}
\psi(a,c,e^{-2\pi i}\ze)=e^{2\pi i a}
\psi(a,c,\ze)-{2\pi i \over \Gamma(a)\Gamma(a-c+1)}
e^{i\pi a}e^\ze \psi(c-a,c,e^{-i\pi}\ze),
\ee
This property is proved in the appendix of \ci{IK} (equation (7.30)).

Taking $\Psi^{(I)}_j(\ze)$ given by (\ref{PsiConfl1}) and applying to
it the jump condition for $\ze<0$, we obtain using (\ref{prop})
and the standard properties of $\Gamma$-function the following expressions
for the first column of $\Psi^{(IV)}$:
\begin{eqnarray}\label{Psi3}
\Psi^{(IV)}_{j,11}(\ze)&=&
\ze^{\al_j}\psi(\al_j+\bt_j,1+2\al_j,e^{-2\pi i}\ze)e^{-\ze/2}\\
\Psi^{(IV)}_{j,21}(\ze)&=&
-\ze^{-\al_j}\psi(1-\al_j+\bt_j,1-2\al_j,e^{-2\pi i}\ze)e^{-i\pi\bt_j}e^{-\ze/2}
{\Gamma(1+\al_j+\bt_j)\over\Gamma(\al_j-\bt_j)}
\end{eqnarray}
The second column is
\be\la{Psi3l}
\Psi^{(IV)}_{j,12}(\ze)=\Psi^{(I)}_{j,12}(\ze)e^{-i\pi\al_j},\qquad
\Psi^{(IV)}_{j,22}(\ze)=\Psi^{(I)}_{j,22}(\ze)e^{-i\pi\al_j}
\ee

Now applying to this function the jump condition
for $\ze>0$ and using again (\ref{prop}), we obtain (note that
as a result of these manipulations we moved $\ze\rightarrow e^{2\pi i}\ze$)
\be
\Psi^{(I)}_j(\ze)=\Psi_j(\ze)
\ee
with $0<\arg\ze<\pi$, i.e. the $\Psi^{(I)}_j(\ze)$ we started with.
Thus, (\ref{PsiConfl1}, \ref{Psi3})
is a solution to the reduced RHP given by the jump condition (\ref{newjump}).
Therefore, (\ref{PsiConfl1}, \ref{Psi3}--\ref{Psi3l})
give a solution to the original RHP
for $\Psi$ in the sectors $I$ and $IV$, respectively; and the solution
in the other sectors is reconstructed using (\ref{JumpPsi1}--\ref{JumpPsi8}).
Proposition \ref{Param} is proved.
\end{proof}

We will now match this solution
with $N(z)$
on the boundary $\partial U_{z_j}$ for large $n$. The limit $n\to\infty$,
$z\in \partial U_{z_j}$,
corresponds to $\ze\to\infty$, therefore
we need the asymptotic expansion of $\Psi_j(\ze)$.
We use the classical result
(e.g., \ci{BE} or Eq.(7.2) of
\ci{IK}) for the confluent hypergeometric function:
\be\la{asp}
\psi(a,c,x)=x^{-a}[1-a(1+a-c)x^{-1}+O(x^{-2})],\qquad |x|\to\infty,\qquad
-3\pi/2<\arg x<3\pi/2.
\ee
Note that these asymptotics can be taken both for $\psi(a,c,\ze)$ and
$\psi(a,c,e^{-i\pi}\ze)$ for $\ze\in I$.
We apply this result to (\ref{PsiConfl1}) and thus obtain the asymptotics of the
solution in the sector $I$. The ``proper'' triangular structure of the jump
matrices implies that these asymptotics
remain the same in the sector $II$ as well, namely:
\begin{multline}\la{Psias}
\Psi^{(I)}_j(\ze)=\Psi^{(II)}_j(\ze)=
\left[ I+{1\over\ze}\bm
\al_j^2-\bt_j^2&
{\Ga(1+\al_j-\bt_j)\over\Ga(\al_j+\bt_j)}e^{i\pi(\bt_j+4\al_j)}\cr
-{\Ga(1+\al_j+\bt_j)\over\Ga(\al_j-\bt_j)}e^{-i\pi(\bt_j+4\al_j)}
&-(\al_j^2-\bt_j^2)
\em
+O(\ze^{-2})\right]\\
\times\ze^{-\bt_j\si_3}e^{-\ze\si_3/2}
\bm
e^{i\pi(2\bt_j+\al_j)} & 0\cr 0 & e^{-i\pi(\bt_j+2\al_j)}
\em,\qquad \ze\to\infty,\quad \ze\in I,II,\qquad \al_j\pm\bt_j\neq -1,-2,\dots
\end{multline}
Furthermore, applying the jump matrices, we obtain the following
asymptotics for $\Psi_j(\ze)$ in
the other sectors (here $\Psi^{(I)}_j(\ze)$ stands for the analytic continuation
of the r.h.s. of (\ref{Psias}) to $0<\arg\ze<2\pi$) as $\ze\to\infty$:
\begin{eqnarray}
\Psi^{(III)}_j(\ze)=\Psi^{(IV)}_j(\ze)&=&
\Psi^{(I)}_j(\ze)e^{i\pi\al_j\si_3},\\
\Psi^{(V)}_j(\ze)=\Psi^{(VI)}_j(\ze)&=&
\Psi^{(I)}_j(\ze)
\bm
0& -e^{i\pi\bt_j}\cr
e^{-i\pi\bt_j}& 0
\em
e^{-i\pi\al_j\si_3},\\
\Psi^{(VII)}_j(\ze)=\Psi^{(VIII)}_j(\ze)&=&
\Psi^{(I)}_j(\ze)
\bm
0& -e^{-i\pi\bt_j}\cr
e^{i\pi\bt_j}& 0
\em.
\end{eqnarray}

Now substituting these asymptotics into the condition on $E$:
\be\la{match2}
P_{z_j}(z)N^{-1}(z)=
E(z)\Psi_j(\ze)F_j(z)^{-\si_3}z^{\pm n\si_3/2}N^{-1}(z)=
I+o(1),
\ee
we obtain
\begin{eqnarray}\la{E1}
&E(z)=N(z)\ze^{\bt_j\si_3} F_j^{\si_3}(z) z_j^{-n\si_3/2}
\bm
e^{-i\pi(2\bt_j+\al_j)} &0\cr 0 & e^{i\pi(\bt_j+2\al_j)}
\em,
\quad
\mbox{for $\ze\in I,II$},\\
&E(z)=N(z)\ze^{\bt_j\si_3} F_j^{\si_3}(z) z_j^{-n\si_3/2}
\bm
e^{-2\pi i(\bt_j+\al_j)} &0\cr 0 & e^{i\pi(\bt_j+3\al_j)}
\em,
\quad
\mbox{for $\ze\in III,IV$},
\end{eqnarray}
\begin{eqnarray}
&E(z)=N(z)\ze^{-\bt_j\si_3} F_j^{\si_3}(z) z_j^{n\si_3/2}
\bm
0& e^{i\pi (3\al_j+2\bt_j)}\cr -e^{-i\pi(3\bt_j+2\al_j)}& 0
\em,
\quad
\mbox{for $\ze\in V,VI$},\\
&E(z)=N(z)\ze^{-\bt_j\si_3} F_j^{\si_3}(z) z_j^{n\si_3/2}
\bm
0& e^{2\pi i\al_j}\cr -e^{-i\pi(\bt_j+\al_j)}& 0
\em,
\quad
\mbox{for $\ze\in VII,VIII$}.\la{E4}
\end{eqnarray}

The dependence on $z$ enters into these expressions only via the combination
$\mathcal{D}(z)/(\ze^{\bt_j}F_j(z))$ for $|z|<1$ (i.e., $\ze\in I,II,III,IV$) and
the combination $\mathcal{D}(z)F_j(z)/\ze^{\bt_j}$ for $|z|>1$
(i.e., $\ze\in V,VI,VII,VIII$). Expanding the logarithm in (\ref{zeta}) in
powers of $u=z-z_j$, we see immediately from (\ref{Dl1},\ref{Dg1},\ref{Fj},\ref{F0})
that the mentioned combinations, and therefore $E(z)$ have no singularity
at $z_j$. Thus $E(z)$ is an analytic function in $U_{z_j}$.
In what follows, we will need more detailed information about the behaviour
of some of these combinations as $u\to 0$.
Namely, it is easy to obtain from (\ref{zeta},\ref{Dl1},\ref{Fj},\ref{F0}) and (\ref{abszzj}) that
\be\la{etab}
F_j(z)=\eta_j e^{-3i\pi\al_j/2}z_j^{-\al_j}
u^{\al_j}(1+O(u)),\qquad u=z-z_j,\qquad \ze\in I,
\ee
where
\be\la{eta}
\eta_j=e^{V(z_j)/2}\exp\left\{-{i\pi\over 2}\left(\sum_{k=0}^{j-1}\bt_k-
\sum_{k=j+1}^m\bt_k\right)\right\}
\prod_{k\neq j}\left({z_j\over z_k}\right)^{\bt_k/2}|z_j-z_k|^{\al_k},
\ee
and
\begin{eqnarray}\la{mub}
\left({\mathcal{D}(z)\over \ze^{\bt_j}F_j(z)}\right)^2=\mu^2_j e^{i\pi(\al_j-2\bt_j)}
n^{-2\bt_j}(1+O(u)),
\qquad u=z-z_j,\qquad \ze\in I,
\\
\mu_j=\left(e^{V_0}\frac{b_+(z_j)}{b_-(z_j)}\right)^{1/2}
\exp\left\{-{i\pi\over2}\left(\sum_{k=0}^{j-1}\al_k-
\sum_{k=j+1}^m\al_k\right)\right\}
\prod_{k\neq j}\left({z_j\over z_k}\right)^{\al_k/2}|z_j-z_k|^{\bt_k}.
\la{mu}
\end{eqnarray}
To derive (\ref{mu}), we used, in particular, the factorization (\ref{WienH}).
The sums from $0$ to $-1$ and from $m+1$ to $m$ are set to zero.

It is seen directly from (\ref{E1}--\ref{E4}) that
$\det E(z)=e^{i\pi(\al_j-\bt_j)}$.
Note that as follows by Liouville's theorem from the RHP,
$\det\Psi_j(\ze)=e^{-i\pi(\al_j-\bt_j)}$:
this function has no jumps, the singularity at zero is removable
as $\Re\al_j>-1/2$, and the constant value follows from the asymptotics
(\ref{Psias}). Combining these results, we see from (\ref{Plb0}) that
$\det P_{z_j}(z)=1$.
Comparing the conditions (\ref{Psic1},\ref{Psic2}) and (\ref{Sd1},\ref{Sd2}),
we see that
the singularity of  $S(z)P_{z_j}(z)^{-1}$ at $z=z_j$ is at most
$O(|z-z_j|^{2\al_j})$ or $O(\ln^2|z-z_j|)$. However, by construction of $P_{z_j}$,
the function $S(z)P_{z_j}(z)^{-1}$ has no jumps  in a neighborhood of $U_{z_j}$
and hence this singularity is removable. Thus,
$S(z)P_{z_j}(z)^{-1}$ is analytic in a neighborhood of $U_{z_j}$.

Note that the error term in (\ref{match2}) $o(1)
=n^{-\Re\bt_j\si_3}O(n^{-1})n^{\Re\bt_j\si_3}$. It is $o(1)$ for
$-1/2<\Re\bt_j<1/2$.

This completes the construction of the parametrix at $z_j$:
it is given by the formulae
(\ref{Plb0},\ref{P1},\ref{E1}--\ref{E4}) and Proposition \ref{Param}.

Considering further terms in (\ref{Psias}), we can extend
(\ref{match2}) into the full asymptotic series in inverse powers of $n$.
For our calculations we need to know explicitly the first correction term:
\begin{multline}\la{az}
P_{z_j}(z)N^{-1}(z)=I+\De_1(z)+
n^{-\Re\bt_j\si_3}O(1/n^2)n^{\Re\bt_j\si_3},\\
\De_1(z)={1\over\ze}
\bm
-(\al_j^2-\bt_j^2)&
{\Ga(1+\al_j+\bt_j)\over \Ga(\al_j-\bt_j)}
\left({\mathcal{D}(z)\over\ze^{\bt_j}F_j(z)}\right)^2
z_j^n e^{i\pi(2\bt_j-\al_j)}\cr
-{\Ga(1+\al_j-\bt_j)\over \Ga(\al_j+\bt_j)}
\left({\mathcal{D}(z)\over\ze^{\bt_j}F_j(z)}\right)^{-2}
z_j^{-n} e^{-i\pi(2\bt_j-\al_j)}&
\al_j^2-\bt_j^2
\em,\\
\qquad z\in\partial z(I),\qquad \al_j\pm\bt_j\neq -1,-2,\dots,
\end{multline}
where $\partial z(I)$ is the
part of $\partial U_{z_j}$ whose $\ze$-image is in $I$.
As a consideration of the other sectors shows,
this expression for $\De_1(z)$ extends by analytic continuation
to the whole boundary $\partial U_{z_j}$.
As follows from (\ref{mub}), it gives a meromorphic function in
a neighborhood of $U_{z_j}$ with a simple pole at $z=z_j$.

The error term $O(1/n^2)$ in (\ref{az}) is uniform in $z$ on $\partial U_{z_j}$.

%%%%%%%%%%% R %%%%%%%%%%%%%

\subsection{R-RHP}\la{RRHP}
Throughout this section we assume that $\al_j\pm\bt_j\neq -1,-2,\dots$ for
all $j=0,1,\dots,m$.

Let
\be
R(z)=\begin{cases}S(z)N^{-1}(z),&
z\in U_\infty\setminus\Gamma,\qquad U_\infty=\bbc\setminus\cup_{j=0}^m U_{z_j},\cr
S(z)P_{z_j}^{-1}(z),&
z\in U_{z_j}\setminus\Gamma,\qquad j=0,\dots, m.
\end{cases}\la{wtR}
\ee
It is easy to verify that this function has jumps only on
$\partial U_{z_j}$, and parts of
$\Si_j$,  $\Si^{''}_j$
lying outside the neighborhoods $U_{z_j}$ (we
denote these parts without the end-points $\Si^\mathrm{out}$, $\Si^{''\mathrm{out}}$).
The full contour $\Gamma$ is shown in Figure 3 where $U_j\equiv U_{z_j}$. 
Away from $\Gamma$, as a standard argument shows, $R(z)$ is analytic.
Moreover, we have: $R(z)=I+O(1/z)$ as $z\to\infty$.

\begin{figure}
\centerline{\psfig{file=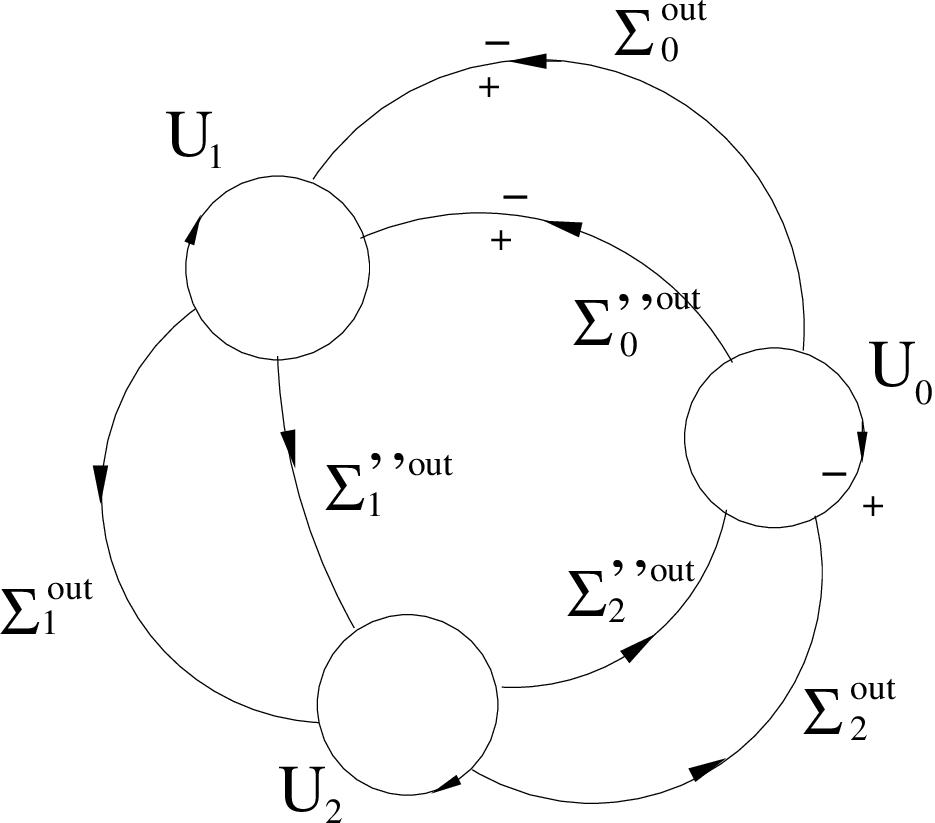,width=3.0in,angle=0}}
\vspace{0cm}
\caption{
Contour $\Gamma$ for the $R$ and $\wt R$ Riemann-Hilbert problems ($m=2$).}
\label{fig3}
\end{figure}

The jumps of $R(z)$ are as follows:
\begin{eqnarray}
R_+(z)&=&R_-(z)N(z)
\bm
1&0\cr f(z)^{-1}z^{-n}&1
\em
N(z)^{-1},\qquad z\in\Si_j^\mathrm{out},\la{Rs1}\\
R_+(z)&=&R_-(z)N(z)
\bm
1&0\cr f(z)^{-1}z^{n}&1
\em
N(z)^{-1},\qquad z\in\Si_j^{''\mathrm{out}},\la{Rs2}\\
R_+(z)&=&R_-(z)P_{z_j}(z)N (z)^{-1},\qquad
z\in \partial U_{z_j}\setminus
\mbox{\{intersection points\}},\\
j=0,\dots,m.\nonumber
\end{eqnarray}

The jump matrix on $\Si^\mathrm{out}$, $\Si^{''\mathrm{out}}$ can be
estimated uniformly in $\al_j$, $\bt_j$ as $I+O(\exp(-\ep n))$, where
$\ep$ is a positive constant.
The jump matrices on $\partial U_{z_j}$ admit a uniform expansion in the
inverse powers of $n$ conjugated by $n^{\bt_j\si_3}z_j^{-n\si_3/2}$
(the first term is given explicitly by (\ref{az})):
\be\la{Deas}
I+\De_1(z)+\De_2(z)+\cdots+\De_{k}(z)+\De^{(r)}_{k+1},
\qquad z\in\partial U_{z_j}.
\ee
Every $\De_p(z)$, $\De^{(r)}_p(z)$, $p=1,2,\dots$,
$z\in\cup_{j=0}^m\partial U_{z_j}$ is of the form
\be\la{ordDe}
\sum_{j=0}^m a_j^{-\si_3}O(n^{-p})a_j^{\si_3},\qquad a_j\equiv n^{\bt_j}
z_j^{-n/2},
\ee
it is of order $n^{2\max_j|\Re\bt_j|-p}$.

To obtain a standard solution of the $R$-RHP in terms of a Neumann series
(see, e.g., \cite{Dstrong}) we must have
$n^{2\max_j|\Re\bt_j|-1}=o(1)$, that is
$\Re\bt_j\in (-1/2,1/2)$ for all $j=0,1,\dots,m$.

However, it is possible to obtain the solution in any open or half-closed interval of length 1,
i.e. for $|||\bt|||<1$.
Namely, let $|||\bt|||<1$ and consider the transformation
\be\la{Rtilde}
\wt R(z)=n^{\om\si_3}R(z) n^{-\om\si_3}\qquad z\in\bbc\setminus\Gamma,
\ee
where
\be\la{omega}
\om={1\over 2}(\min_j\Re\bt_j+\max_j\Re\bt_j)
\ee
will ``shift'' all $\Re\bt_j$ (in the conjugation $n^{\bt_j}$ terms)
inside the interval $(-1/2,1/2)$.
Note that $\om=\Re\bt_{j_0}$ if only one $\Re\bt_{j_0}\neq 0$, and $\om=0$
if all $\Re\bt_j=0$.

In the RHP for $\wt R(z)$, the condition at infinity and the uniform
exponential estimate $I+O(\exp(-\ep n))$ (with different $\ep$)
of the jump matrices on $\Si^\mathrm{out}$, $\Si^{''\mathrm{out}}$ are preserved,
while the jump matrices on $\partial U_{z_j}$ have the form:
\be\la{Deas2}
I+n^{\om\si_3}\De_1(z)n^{-\om\si_3}+\cdots+
n^{\om\si_3}\De_{k}(z)n^{-\om\si_3}+
n^{\om\si_3}\De_{k+1}^{(r)}(z)n^{-\om\si_3},\qquad
z\in\partial U_{z_j},
\ee
where the order of each $n^{\om\si_3}\De_p(z)n^{-\om\si_3}$,
$n^{\om\si_3}\De_p^{(r)}(z)n^{-\om\si_3}$, $p=1,2,\dots$,
$z\in\cup_{j=0}^m\partial U_{z_j}$
is
\[
O(n^{2\max_j|\Re\bt_j-\om|-p}).
\]
This implies that the standard analysis can be applied to the $\wt R$-RHP problem
in the range $\Re\bt_j\in (q-1/2,q+1/2)$, $j=0,1,\dots,m$,
for any $q\in\bbr$, and we obtain the asymptotic expansion
\be\la{Ras}
\wt R(z)=I+\sum_{p=1}^{k}\wt R_p(z)+\wt R^{(r)}_{k+1}(z),\quad k=1,2,\dots
\ee
In our case the error term
\be\la{Rerr}
\wt R^{(r)}_{k+1}(z)=O(|\wt R_{k+1}(z)|)+O(|\wt R_{k+2}(z)|).
\ee

The functions $\wt R_j(z)$ are computed recursively.
In this paper, we will need explicit expressions only for
the first two. Accordingly, set $k=2$.
The function  $\wt R_1(z)$ is found from the conditions
that it is analytic outside
$\partial U=\cup_{j=0}^m\partial U_{z_j}$,
$\wt R_1(z)\to 0$ as $z\to \infty$, and
\be
\wt R_{1,+}(z)=\wt R_{1,-}(z)+n^{\om\si_3}\De_1(z)n^{-\om\si_3},\qquad
z\in\partial U.
\ee
The solution is easily written. First denote
\be
R_p(z)\equiv n^{-\om\si_3}\wt R_p(z)n^{\om\si_3},\qquad
R_p^{(r)}(z)\equiv n^{-\om\si_3}\wt R_p^{(r)}(z)n^{\om\si_3},
\ee
and write for $R$:
\begin{multline}
R_1(z)=
{1\over 2\pi i}\int_{\partial U}
{\De_1(x)dx\over x-z}\\
=\begin{cases}
\sum_{k=0}^m {A_k\over z-z_k},& z\in\bbc\setminus \cup_{j=0}^m U_{z_j}\cr
\sum_{k=0}^m {A_k\over z-z_k}-\De_1(z),& z\in U_{z_j},\quad j=0,1,\dots, m.
\end{cases},\qquad
\partial U=\cup_{j=0}^m\partial U_{z_j}.
\la{plem}
\end{multline}
where the contours in the integral are traversed in the negative direction,
and $A_k$ are the coefficients in the Laurent expansion of $\De_1(z)$:
\be
\De_1(z)={A_k\over z-z_k}+B_k+O(z-z_k),\qquad z\to z_k,\qquad k=0,1,\dots,m.
\ee
The coefficients are easy to write using (\ref{az}) and (\ref{mub}):
\be\la{A}
A_k\equiv A^{(n)}_k={z_k\over n}
\bm
-(\al_k^2-\bt_k^2)&
{\Ga(1+\al_k+\bt_k)\over \Ga(\al_k-\bt_k)}
z_k^n \mu^2_k n^{-2\bt_k}\cr
-{\Ga(1+\al_k-\bt_k)\over \Ga(\al_k+\bt_k)}
z_k^{-n} \mu_k^{-2}n^{2\bt_k}&
\al_k^2-\bt_k^2
\em.
\ee
An expression for $B_k$ is also easy to find, but it is not needed below.

The function $\wt R_2$ is now found from the conditions that $\wt R_2(z)\to 0$
as $z\to\infty$, is analytic outside $\partial U$, and
\be
\wt R_{2,+}(z)=\wt R_{2,-}(z)+\wt R_{1,-}(z)n^{\om\si_3}\De_1(z)n^{-\om\si_3}+
n^{\om\si_3}\De_2(z)n^{-\om\si_3} ,\qquad
z\in\partial U.
\ee
The solution to this RHP is
\be\la{R2}
\wt R_2(z)=
 {1\over 2\pi i}\int_{\partial U}\left(
\wt R_{1,-}(x)n^{\om\si_3}\De_1(x)n^{-\om\si_3}+
n^{\om\si_3}\De_2(x)n^{-\om\si_3}
\right){dx\over x-z}.
\ee

Further standard analysis (cf. (\ref{Rerr})) shows that the error term
\be\la{orderr}
R^{(r)}_3(z)=
\begin{pmatrix}
O(\de/n)+O(\de^2) & O\left(\de\max_k{n^{-2\Re\bt_k}\over n}\right)\cr
O\left(\de\max_k{n^{2\Re\bt_k}\over n}\right) & O(\de/n)+O(\de^2)
\end{pmatrix},
\ee
where $\de$ is given by (\ref{de}).

In particular, as is clear from the above, if there is only one
nonzero $\bt_{j_0}$, we obtain the expansion of $\wt R(z)$
purely in inverse integer powers of $n$ valid in fact for all $\bt_{j_0}\in \bbc$,
$\al_{j_0}\pm\bt_{j_0}\neq -1,-2,\dots$

It is clear from the construction and the properties of the asymptotic series
of the confluent hypergeometric function that
the error terms $\wt R^{(r)}_k(z)$ (\ref{Rerr}), and in particular (\ref{orderr}),
are uniform for all $z$ and for $\bt_j$ in bounded sets of the strip
$q-1/2<\Re\bt_j< q+1/2$, $j=0,1,\dots m$,
for $\al_j$ in bounded sets of the half-plane
$\Re\al_j>-1/2$, and for $\al_j\pm\bt_j$ away from neighborhoods of the negative integers.
Moreover, the series (\ref{Ras}) is differentiable in $\al_j$, $\bt_j$.

For future use note that if $V(z)=V_r(z)+(V(z)-V_r(z))h$, $h\in[0,1]$, and $V_r(z)$ is analytic
in a neighborhood of the unit circle, then the error terms are uniform in the parameter
$h\in[0,1]$.

\section{Orthogonal polynomials. Proof of Theorem \ref{poly}}
Using results of the previous section, we can provide a complete
asymptotic analysis of the polynomials orthogonal with weight (\ref{fFH})
on the unit circle with analytic $V(z)$. In this section we will find the
asymptotic expressions for $\chi_n$, $\phi_n(0)$, and $\widehat\phi_n(0)$.

First, it follows immediately from (\ref{RHM}) that
\be
\chi_{n-1}^2=-Y^{(n)}_{21}(0).
\ee
Tracing back the transformations $R\rightarrow S\rightarrow T
\rightarrow Y$, we obtain for $z$ inside the unit circle and outside the lenses:
\begin{multline}
Y(z)=T(z)=S(z)=R(z) N(z)=
n^{-\om\si_3}\wt R(z)n^{\om\si_3}N(z)\\=
n^{-\om\si_3}[I+\wt R_1(z)+\wt R_2(z)+\wt R^{(r)}_3(z)]n^{\om\si_3}N(z)\\=
\left[I+R_1(z)+R_2(z)+
R^{(r)}_3(z)
\right]\mathcal{D}(z)^{\si_3}
\bm
0&1\cr -1&0
\em.
\end{multline}
Taking the $21$ matrix element and setting $z=0$ we obtain
\be\la{chimid}
\chi_{n-1}^2=-Y^{(n)}_{21}(0)=
\mathcal{D}(0)^{-1}\left[1+R_{1,22}(0)+R_{2,22}(0)+O(\de/n+\de^2)
\right],
\ee
where we used the estimate (\ref{orderr}) for $R^{(r)}_3(z)$.

By (\ref{Dl1})
\be
\mathcal{D}(0)^{-1}=
\exp\left[-\int_0^{2\pi}V(\th){d\th\over2\pi}\right]=e^{-V_0}.
\ee
Using (\ref{plem}) and (\ref{A}) we obtain
\be
R_{1,22}(0)=-\sum_{k=0}^m {A_{k,22}\over z_k}=
-{1\over n}\sum_{k=0}^m\left(\al_k^2-\bt_k^2\right).
\ee

Conjugating (\ref{R2}) with $n^{\om\si_3}$, setting there
$z=0$, and applying (\ref{plem}),
we obtain:
\be
R_2(0)= -\sum_{j=0}^m z_j^{-1} \sum_{k\neq j}{A_kA_j\over z_j-z_k}+
{1\over n^2}O
\bm
1& \sum_j n^{-2\bt_j}\cr
\sum_j n^{2\bt_j} & 1
\em.
\ee
From (\ref{A}),
\begin{multline}
(A_kA_j)_{22}=z_k z_j\left[(\al_j^2-\bt_j^2)(\al_k^2-\bt_k^2)n^{-2}
\right.\\
\left.
-n^{2(\bt_k-\bt_j-1)}\left({z_j\over z_k}\right)^n
{\Ga(1+\al_j+\bt_j)\Ga(1+\al_k-\bt_k)\over
\Ga(\al_j-\bt_j)\Ga(\al_k+\bt_k)}{\mu^2_j\over\mu^2_k}
\right],
\end{multline}
where $\mu^2_j$ are defined in (\ref{mu}).

Substituting the last 3 equations into (\ref{chimid}), we finally obtain
(\ref{aschi}).

We now turn our attention to $\phi_n(0)$. Using (\ref{RHM}), we have
\be\la{phi0}
\phi_n(0)=\chi_n Y_{11}^{(n)}(0)=
\chi_n\left(R(0) \mathcal{D}(0)^{\si_3}
\bm
0&1\cr
-1&0
\em
\right)_{11}=
-\chi_n\mathcal{D}(0)^{-1}\left( R_{1,12}(0)+
R^{(r)}_{2,12}(0)\right).
\ee
By (\ref{plem},\ref{A}),
\be\la{phi0r}
R_{1,12}(0)=-\sum_{k=0}^m {A_{k,12}\over z_k}=
-{1\over n}\sum_{j=0}^m n^{-2\bt_j}z_j^n
{\Ga(1+\al_j+\bt_j)\over\Ga(\al_j-\bt_j)}\mu^2_j,
\ee
and, recalling (\ref{Rerr}), we obtain (\ref{asphi}).

Finally, starting again with (\ref{RHM}), we have
\begin{multline}
\widehat\phi_{n-1}(0)=-\frac{1}{\chi_{n-1}}\lim_{z\to\infty}
\frac{Y_{21}^{(n)}(z)}{z^{n-1}}=
-\frac{1}{\chi_{n-1}}\lim_{z\to\infty}{1\over z^{n-1}}(R(z)\mathcal{D}(z)^{\si_3}
z^{n\si_3})_{21}\\
=
-\frac{1}{\chi_{n-1}}\left(\lim_{z\to\infty}zR_{1,21}(z)+
O\left(\left[\de+{1\over n}\right]\max_k{n^{2\Re\bt_k}\over n}
\right)\right).
\end{multline}
We have
\be
\lim_{z\to\infty}zR_{21}(z)=
\sum_{k=0}^m A_{k,21}=
-{1\over n}\sum_{j=0}^m n^{2\bt_j}z_j^{-n+1}
{\Ga(1+\al_j-\bt_j)\over\Ga(\al_j+\bt_j)}\mu_j^{-2},
\ee
and therefore, recalling (\ref{aschi}), obtain (\ref{ashatphi}).

Note that uniformity and differentiability properties of the asymptotic
series of Theorem \ref{poly} follow from those
of the $\wt R$-expansion of the previous section.

\section{Toeplitz determinants. Proof of Theorem \ref{BT}}\la{secBT}

\subsection{The case of analytic $V(z)$}\la{secBTa}
First, let $V(z)$ be analytic in a neighborhood of the unit circle.
Consider the set
$\bt^{(r)}_j$ constructed in the proof of Lemma \ref{repr}. We have to consider only
the second class, i.e. $|||\bt^{(r)}|||=1$. We then have, relabeling
$\bt^{(r)}_j$ according to increasing real part,
\be
\Re\bt^{(r)}_1=\cdots=\Re\bt^{(r)}_p<\Re\bt^{(r)}_{p+1}\le\cdots
\le\Re\bt^{(r)}_{m'-\ell}<
\Re\bt^{(r)}_{m'-\ell+1}=\cdots=\Re\bt^{(r)}_{m'},
\ee
for some $p,\ell>0$. Here $m'$ is the number of singularities: $m'=m+1$
if $z=1$ is a singularity ($\al_0\neq 0$ or $\bt_0\neq 0$), otherwise $m'=m$.
Now consider the symbol (not a FH-representation of $f$)
$\wt f$ of type (\ref{fFH})
with beta-parameters denoted by $\wt\bt$ and given by
$\wt\bt_j=\bt^{(r)}_j$ for $j=1,\dots, m'-\ell$, and
$\wt\bt_j=\bt^{(r)}_j-1$ for $j=m'-\ell+1,\dots, m'$.
It is easy to see that the original symbol $f$ has ${\ell+p\choose\ell}$
FH-representations in $\mathcal{M}$ obtained by shifting any $\ell$
out of $\ell+p$ parameters
$\wt\bt_j$, say $\wt\bt_{i_1},\dots,\wt\bt_{i_\ell}$,
with the smallest real part to the right by $1$.
Thus,
\be\la{ffin}
f(z)=(-1)^\ell\prod_{j=0}^{m}z_j^{L_j}\times z_{i_1}^{-1}\cdots z_{i_\ell}^{-1} z^\ell \wt f(z),
\ee
for appropriate $L_j$.

Let us now and until the end of this section relabel
$\wt\bt_j$, $\al_j$, $L_j$, and $z_j$ according to increasing real
part of $\wt\bt_j$. Thus, in particular,
\be\la{numbering}
\Re\wt\bt_1=\cdots=\Re\wt\bt_{\ell+p}<\Re\wt\bt_{\ell+p+1}.
\ee
Assume that the set of all the minimizing FH-representations
$\mathcal{M}$ is non-degenerate (see Introduction). This implies that
$\al_j\pm\wt\bt_j\neq -1,-2,\dots$.

We now apply Lemma \ref{Chr} (equation (\ref{213}))
to finish the proof of Theorem \ref{BT}. We need to evaluate
the determinant $F_n$, $n\ge N_0$ for a sufficiently large $N_0>0$. First, from (\ref{RHM}),
tracing back the transformations of the RH problem and using (\ref{plem},\ref{A})
we obtain (cf. (\ref{phi0},\ref{phi0r})) for the polynomials orthonormal with weight $\wt f(z)$:
\be\la{phiz}
\phi_n(z)/\chi_n=
\mathcal{D}(z)^{-1}\rho_n(z),\qquad
\rho_n(z)=
-\sum_{k=0}^m {A^{(n)}_{k,12}\over z-z_k}
+O\left(\left[\de+{1\over n}\right]n^{-2\Re\wt\bt_1-1}
\right).
\ee
This expansion is uniform and differentiable in a neighborhood of zero.
A simple algebra shows that in the determinant
\be
F_n=
\left|
\begin{matrix}
\phi_n(0)/\chi_n& \phi_{n+1}(0)/\chi_{n+1} & \cdots &
\phi_{n+\ell-1}(0)/\chi_{n+\ell-1}\cr
{d\over dz}\phi_n(0)/\chi_n & {d\over dz}\phi_{n+1}(0)/\chi_{n+1}&
\cdots & {d\over dz}\phi_{n+\ell-1}(0)/\chi_{n+\ell-1}\cr
\vdots & \vdots &  & \vdots\cr
 {d^{\ell-1}\over dz^{\ell-1}}\phi_n(0)/\chi_n &
{d^{\ell-1}\over dz^{\ell-1}}\phi_{n+1}(0)/\chi_{n+1}& \cdots &
{d^{\ell-1}\over dz^{\ell-1}}\phi_{n+\ell-1}(0)/\chi_{n+\ell-1}
\end{matrix}
\right|
\ee
all the terms with the derivatives of $\mathcal{D}(z)$ drop out, and
we have
\be
F_n=\mathcal{D}(0)^{-\ell}
\left|
\begin{matrix}
\rho_n(0)& \rho_{n+1}(0) & \cdots &
\rho_{n+\ell-1}(0)\cr
{d\over dz}\rho_n(0) & {d\over dz}\rho_{n+1}(0)&
\cdots & {d\over dz}\rho_{n+\ell-1}(0)\cr
\vdots & \vdots &  & \vdots\cr
 {d^{\ell-1}\over dz^{\ell-1}}\rho_n(0) &
{d^{\ell-1}\over dz^{\ell-1}}\rho_{n+1}(0)& \cdots &
{d^{\ell-1}\over dz^{\ell-1}}\rho_{n+\ell-1}(0)
\end{matrix}
\right|.
\ee
It is a crucial fact
that the size $\ell$ of this determinant is less than the number
of terms, $\ell+p$, in the expansion of $\phi_n(0)/\chi_n$ of the same
largest order $O(n^{-2\Re\wt\bt_1-1})$ (see (\ref{asphi}) with $\bt_j$
replaced by $\wt\bt_j$).

As $|\Re\wt\bt_j-\Re\wt\bt_k|<1$,
and $\al_j\pm\wt\bt_j\neq -1,-2,\dots$, $j,k=1,\dots, m'$,
we obtain for the $p$'th derivative of $\rho(z)$ from (\ref{phiz}), (\ref{A}), and
(\ref{mu}) with $\bt$ replaced by $\wt\bt$
\begin{multline}\la{rhoj}
{d^s\over dz^s}\rho_{n+i}(0)=s!\sum_{k=0}^m {A^{(n+i)}_{k,12}\over z_k^{s+1}}
+O\left(\left[\de+{1\over n}\right]n^{-2\Re\wt\bt_1-1}\right)\\
=s!\sum_{j=1}^{\ell+p}d_j z_j^{n+i-s}+O\left(n^{-2\Re\wt\bt_{\ell+p+1}-1}\right)
+O\left(\left[\de+{1\over n}\right]n^{-2\Re\wt\bt_1-1}\right),
\end{multline}
where
\be\la{dj}
d_j= n^{-2\wt\bt_j-1}
{\Ga(1+\al_j+\wt\bt_j)\over\Ga(\al_j-\wt\bt_j)}
e^{V_0}{b_+(z_j)\over b_-(z_j)}\exp\left\{-{i\pi}\left(\sum_{k=0}^{j-1}\al_k-
\sum_{k=j+1}^m\al_k\right)\right\}
\prod_{k\neq j}\left({z_j\over z_k}\right)^{\al_k}|z_j-z_k|^{2\wt\bt_k}.
\ee
Substituting these expressions into the determinant $F_n$, we obtain
\begin{multline}
F_n=\mathcal{D}(0)^{-\ell}
\prod_{s=0}^{\ell-1}s!
\sum_{1\le i_1\neq i_2\neq\cdots\neq i_\ell\le \ell+p}
d_{i_1}d_{i_2}\cdots d_{i_\ell}
z_{i_1}^n\cdots z_{i_\ell}^{n-\ell+1}
\prod_{1\le j<k\le\ell}(z_{i_k}-z_{i_j})(1+o(1))\\
=\mathcal{D}(0)^{-\ell}
\prod_{s=0}^{\ell-1}s!
\sum_{1\le i_1<i_2<\cdots<i_\ell\le \ell+p}
d_{i_1}d_{i_2}\cdots d_{i_\ell}
(z_{i_1}\cdots z_{i_\ell})^{n}
\prod_{1\le j<k\le\ell}|z_{i_j}-z_{i_k}|^2(1+o(1)),
\end{multline}
as $z_j^{-1}=\overline{z_j}$.

Therefore, by (\ref{213}),
\begin{multline}\la{610}
D_n(z^\ell \wt f(z))={(-1)^{n\ell} F_n \over \prod_{s=0}^{\ell-1}s!}
D_n(\wt f(z))\\
=(-1)^{n\ell}
\mathcal{D}(0)^{-\ell}\sum_{1\le i_1<i_2<\cdots<i_\ell\le \ell+p}
d_{i_1}d_{i_2}\cdots d_{i_\ell}
(z_{i_1}\cdots z_{i_\ell})^{n}
\prod_{1\le j<k\le\ell}|z_{i_j}-z_{i_k}|^2 D_n(\wt f(z))
(1+o(1)).
\end{multline}
We now use Theorem \ref{asTop} for $D_n(\wt f(z))$.
Noting, in particular, that $G(1+z)=\Ga(z) G(z)$ and $\mathcal{D}(0)=e^{V_0}$,
we obtain after a straightforward calculation that
\[
D_n(z^\ell \wt f(z))= (-1)^{n\ell}
\sum_{1\le i_1<i_2<\cdots<i_\ell\le \ell+p}
(z_{i_1}\cdots z_{i_\ell})^{n}
\mathcal{R}(\dots,\wt\bt_{i_1}+1,\wt\bt_{i_2}+1,\dots, \wt\bt_{i_\ell}+1,\dots),
\]
where $\mathcal{R}$ is the r.h.s. of (\ref{asD}) where all $\bt_j$ are replaced
with $\wt\bt_j$ with the exception of  $\bt_j$, $j=i_1,\dots,i_\ell$ which are replaced
as indicated in the argument of $\mathcal{R}$.
Note once again that each sum is over indices in the range
$1,\dots, \ell+p$ and we use a special numbering
of indices (cf. (\ref{numbering})).
Finally, recalling (\ref{ffin}), we obtain
\be
D_n(f(z))=
\left(\prod_{j=0}^{m}z_j^{L_j}\right)^n
\sum_{1\le i_1<i_2<\cdots<i_\ell\le \ell+p}
\mathcal{R}(\dots,\wt\bt_{i_1}+1,\wt\bt_{i_2}+1,\dots, \wt\bt_{i_\ell}+1,\dots),
\ee
which is the statement of Theorem \ref{BT} for $V(z)$ analytic in a neighborhood of the
unit circle.

\subsection{Extension to smooth $V(z)$}\la{secBTb}
If $V(z)$ is just sufficiently smooth, in particular $C^\infty$,
 on the unit circle $C$ so that (\ref{Vcond})
holds for $s$ from zero up to and including
some $s\ge 0$,
we can approximate $V(z)$ by trigonometric polynomials
$V^{(n)}(z)=\sum_{k=-p(n)}^{p(n)} V_k z^k$, $z\in C$.
First, consider the case when $|||\bt|||=\max_{j,k}|\Re\bt_j-\Re\bt_k|=2\max_j |\Re\bt_j-\om|<1$,
where $\om$ is defined by (\ref{omega}). (The indices $j,k=0$ are omitted if $\al_0=\bt_0=0$.)
We set
\be\la{pnu}
p=[n^{1-\nu}],\qquad \nu=2\max_j |\Re\bt_j-\om|+\ep_1=|||\bt|||+\ep_1,
\ee
where $\ep_1>0$ is chosen sufficiently small so that $\nu<1$
(square brackets denote the integer part).

First, we need to extend the RH analysis of the previous sections to symbols which
depend on $n$, namely to the case when $V$ in $f$ is replaced by $V^{(n)}$. (We will denote
such $f$ by $f(z,V^{(n)})$, and the original one, by $f(z,V)$.)
We need to have a suitable estimate for the behaviour of the error term in asymptotics with $n$.
For a fixed $f$, our analysis depended, in particular, on the fact that $f(z)^{-1}z^{-n}$ is of order
$e^{-\ep n}$, $\ep>0$, for $z\in \Si^\mathrm{out}$ (see Section \ref{RRHP}), and similarly,
$f(z)^{-1}z^n=O(e^{-\ep n})$ for  $z\in \Si^{''\mathrm{out}}$.
Here the contours $\Si^\mathrm{out}$, $\Si^{''\mathrm{out}}$ are outside a {\it fixed} neighborhood of
the unit circle (outside and inside $C$, respectively).
If $V$ is replaced by $V^{(n)}$, let us define the curve $\Si$ outside $\cup_{j=0}^m U_j$,
$U_j\equiv U_{z_j}$, by
\be\la{si1}
z=\left(1+\ga\frac{\ln p}{p}\right)e^{i\th},\qquad \ga>0,
\ee
and $\Si^{''}$ outside $\cup_{j=0}^m U_j$ by
\be\la{si2}
z=\left(1-\ga\frac{\ln p}{p}\right)e^{i\th}.
\ee
Inside all $U_j$, the curves still go to $z_j$ as discussed in Section \ref{RHparam}.
Let the radius of all $U_j$ be $2\ga\ln p/p$.
We now fix the value of $\ga$ as follows.
Using the condition (\ref{Vcond}) we can write (here and below $c$ stands for various
positive constants independent of $n$)
\begin{multline}
|V^{(n)}(z)|-|V_0|
\le\sum_{k=-p,\;k\neq 0}^p |k^s V_k| \frac{|z|^k}{|k|^s}
< c\left(\sum_{k=-p,\;k\neq 0}^p |k^s V_k|^2\right)^{1/2}
\left(\sum_{k=1}^p\frac{(1\pm 3\ga\ln p/p)^{\pm2k}}{k^{2s}}\right)^{1/2}\\
<c\left(\sum_{k=1}^p\frac{(1\pm 3\ga\ln k/k)^{\pm2k}}{k^{2s}}\right)^{1/2}<
c\left(\sum_{k=1}^p{1\over k^{2(s-3\ga)}}\left[1+O\left({\ln^2 k\over k}\right)\right]\right)^{1/2},
\end{multline}
where $z\in\Si^\mathrm{out}$, $z\in\partial U_j\cap\{|z|>1\}$
(with ``$+$'' sign in ``$\pm$''), and $z\in \Si^{''\mathrm{out}}$,
$z\in\partial U_j\cap\{|z|<1\}$
(with ``$-$'' sign). We now set
\be\la{gas}
3\ga=s-(1+\ep_2)/2,\qquad \ep_2>0,
\ee
and then
\be\la{Vbound}
|V^{(n)}(z)|<c,\qquad |b_+(z,V^{(n)})|<c,\qquad |b_-(z,V^{(n)})|<c,\qquad
\mbox{for all } n
\ee
uniformly on $\Si^\mathrm{out}$, $\Si^{''\mathrm{out}}$, $\partial U_j$'s, and
in fact in the whole annulus $1-3\ga\frac{\ln p}{p}<|z|<1+3\ga\frac{\ln p}{p}$.

It is easy to adapt the considerations of the previous sections to the present case, and
we again obtain the expansion (\ref{Deas}) for the jump matrix of $R$ on $\partial U_j$.
Note that now $|z-z_j|=2\ga(1-\nu)\ln n/ n^{1-\nu}$ and
$|\ze(z)|=2\ga(1-\nu)n^\nu\ln n(1+o(1))$ as $n\to\infty$
for $z\in\partial U_j$, and therefore using (\ref{az}), (\ref{Fj}), (\ref{F0}), (\ref{abszzj}),
(\ref{Dl1}) and the definition of $\nu$ in (\ref{pnu}), we obtain, in particular,
\be
n^{\om\si_3}\De_1(z)n^{-\om\si_3}=O\left({1\over n^{\ep_1}\ln n}\right),\qquad
z\in\cup_{j=0}^m\partial U_j.
\ee
Furthermore, as follows from (\ref{si1}), (\ref{si2}), (\ref{Vbound}),
and (\ref{Rs1},\ref{Rs2}),
the jump matrix on $\Si^\mathrm{out}$ and $\Si^{''\mathrm{out}}$ is now the identity plus a function
uniformly bounded in absolute value by
\be\la{bbb}
c\left({n^{1-\nu}\over\ln n}\right)^{2\max_j|\Re\bt_j|}
\left(1\pm\ga(1-\nu)\frac{\ln n}{n^{1-\nu}}\right)^{\mp n}<
c\exp\left\{-{\ga\over2}(1-\nu)n^\nu\ln n\right\}n^{2(1-\nu)\max_j|\Re\bt_j|},
\ee
where the upper sign corresponds to  $\Si^\mathrm{out}$, and the lower, to $\Si^{''\mathrm{out}}$.

The RH problem for $R(z)$ (see Section \ref{RRHP})
is therefore solvable, and we obtain $R(z)$ as a series where the first term $R_1$ is the same
as before. For the error term there holds the same estimate
for $z$ outside a fixed neighborhood of the unit circle, e.g., at $z=0$.

This, in particular, implies that the formulae (\ref{rhoj},\ref{dj}) hold for $\wt f(z,V^{(n)})$
(in $\wt f$, we substitute $\wt\bt_j$ for $\bt_j$: note that the condition $0<\nu<1$ is satisfied).

We will now show that replacing $V^{(n)}$ with $V$ in the symbol of the determinant
$D_n(\wt f(z,V^{(n)}))$
results (under a condition on $s$) in a small error only, so that (\ref{610}) still holds with
$V$ used in $D_n(\wt f(z,V))$, and  $V^{(n)}$ in $d_j$'s and in  $D_n(z^\ell\wt f(z,V^{(n)}))$.
Then, proceeding as before, we obtain the statement of the theorem for $D_n(z^\ell\wt f(z,V^{(n)}))$
as, by (\ref{Vcond}), uniformly on $C$,
\be\la{bb}
b_\pm(z, V^{(n)})=b_\pm(z, V)\left[1+O\left({1\over n^{(1-\nu)s}}\right)\right],\qquad |z|=1,
\ee
in (\ref{dj}).
Recall a standard representation for a Toeplitz determinant with (any) symbol $f(z)$:
\be\label{multintrepr}
D_{n}(f) = \frac{1}{(2\pi)^{n}n!}
\int_0^{2\pi}\cdots\int_0^{2\pi}
\prod_{1\leq j < k \leq n}|e^{i\phi_{j}} - e^{i\phi_{k}}|^{2}
\prod_{j=1}^n f(e^{i\phi_j})d\phi_{j}.
\ee
We have from this formula, (\ref{bb}), and Theorem \ref{asTop} for $D_n(|\wt f(z,V)|)$
and $D_n(\wt f(z,V))$, if $s(1-\nu)>1$, 
\begin{multline}\la{estD}
\left|D_n(\wt f(z,V))-D_n(\wt f(z,V^{(n)})\right|<\\
\frac{1}{(2\pi)^{n}n!}\int_0^{2\pi}\cdots\int_0^{2\pi}
\prod_{1\leq j < k \leq n}|e^{i\phi_{j}} - e^{i\phi_{k}}|^{2}
\prod_{j=1}^n |\wt f(e^{i\phi_j},V)|d\phi_{j}\times
\left(\left|1+c/n^{(1-\nu)s}\right|^n-1\right)\\
<c e^{\Re V_0 n} n^{\sum_{j=1}^m((\Re\al_j)^2+(\Im\wt\bt_j)^2)}
(e^{c/n^{(1-\nu)s-1}}-1)\\
<c \left|e^{V_0 n} n^{\sum_{j=1}^m (\al_j^2-\wt\bt_j^2)}\right|
n^{\sum_{j=0}^m((\Im\al_j)^2+(\Re\wt\bt_j)^2)}
{1\over n^{(1-\nu)s-1}}\\
<c\left|D_n(\wt f(z,V))\right|n^{-((1-\nu)s-1-\sum_{j=0}^m((\Im\al_j)^2+(\Re\wt\bt_j)^2))}.
\end{multline}
Therefore,
\be\la{estD2}
D_n(\wt f(z,V^{(n)}))=D_n(\wt f(z,V))\left(1+\frac{D_n(\wt f(z,V^{(n)}))-D_n(\wt f(z,V))}
{D_n(\wt f(z,V))}\right)=D_n(\wt f(z,V))(1+o(1)),
\ee
if
\be\la{s-est}
s>\frac{1+\sum_{j=0}^m((\Im\al_j)^2+(\Re\wt\bt_j)^2)}{1-\nu}.
\ee
Under the condition (\ref{s-est}) and the ones under which Theorem \ref{asTop} holds,
$\al_j\pm\wt\bt_j\neq-1,-2,\dots$, and e.g. $C^\infty$ for $V$
(see Remark \ref{smoothness1}),
we then obtain the statement of the theorem for
$D_n(z^\ell\wt f(z,V^{(n)}))$ as mentioned above. The theorem (with Remark \ref{smoothness2})
for  $D_n(z^\ell\wt f(z,V))$, and hence for $D_n(f(z,V))$, immediately follows from
an analysis similar to (\ref{estD},\ref{estD2}) applied to
\[
D_n(z^\ell\wt f(z,V))=D_n(z^\ell\wt f(z,V^{(n)}))
\left(1-\frac{D_n(z^\ell\wt f(z,V^{(n)}))-D_n(z^\ell\wt f(z,V))}
{D_n(z^\ell\wt f(z,V^{(n)}))}\right).
\]
The ratio in the brackets is $o(1)$ under the condition (\ref{s-est})
in which $\wt\bt_j$ are replaced by $\bt^{(r)}_j$ (and the condition under which Theorem \ref{asTop} holds).
As $\ep_1$ can be arbitrary close to zero,
this condition together with (\ref{s-est})
(note that these conditions are consistent with (\ref{gas}) and the requirement that $\ga>0$)
and (\ref{s-main0}) for Theorem \ref{asTop} yield the estimate (\ref{s-main}).

\section{Hankel determinants. Proof of Theorem \ref{asHankel}}
\la{Hankel}
Consider the Hankel determinant with symbol $w(x)$ on $[-1,1]$ given by (\ref{wFH}).
In this section we will find its asymptotics using the relation to
a Toeplitz determinant established in Theorem \ref{HT}.
Let $x=\cos\th$, $z=e^{i\th}$, $0\le\th\le\pi$. In particular,
\[
\lb_j=\cos\th_j,\quad
z_j=e^{i\th_j},\quad j=0,1,\dots,r+1,\quad
0=\th_0<\th_1<\cdots<\th_{r+1}=\pi.
\]
First, we find an even function $f$ of the angle $\th$ related to
$w(x)$ by (\ref{wf}).
The Toeplitz determinant $D_{2n}(f(z))$ with this symbol enters the connection formula
(\ref{HTdet}).
Denote
\be
z'_j=e^{(2\pi -\th_j)i},\qquad j=0,\dots,r+1.
\ee
Then, recalling (\ref{za}), note that
\be
|x-\lb_j|^{2\al_j}=|\cos\th-\cos\th_j|^{2\al_j}=\left|2\sin{\th-\th_j\over 2}
\sin{\th+\th_j\over 2}\right|^{2\al_j}=2^{-2\al_j}|z-z_j|^{2\al_j}
|z-z'_j|^{2\al_j},
\ee
and
\be
|\sin\th|=2^{-1}|z-z_0||z-z_{r+1}|.
\ee
We see that $f(z)$ will have $m+1=2r+2$ singularities at the points
$z_0=1$, $z_{r+1}=-1$, $z_j$, $z'_j$, $j=1,\dots,r$.

Observe that
\be\la{omm}
\prod_{j=1}^r \om_j(x)=
e^{-i\pi\sum_{j=1}^r\bt_j}
\prod_{j=1}^r z_j^{-\bt_j}{z'}_j^{\bt_j}
\prod_{j=1}^r
g_{z_j,-\bt_j}(z)z_j^{\bt_j} g_{z'_j,\bt_j}(z){z'}_j^{-\bt_j}.
\ee

Note that $\bt_0=\bt_{r+1}=0$ and we have the jumps with $-\bt_j$ at $z_j$ and
$+\bt_j$ at $z'_j$. In particular, the sum over all $\bt$'s is zero as noted
in the introduction.
Note that as $\th_j=\pi/2-\arcsin\lb_j$, we have in (\ref{omm})
\be
e^{-i\pi\sum_{j=1}^r\bt_j}
\prod_{j=1}^r z_j^{-\bt_j}{z'}_j^{\bt_j}=
\exp\left(2i\sum_{j=1}^r\bt_j\arcsin\lb_j\right).
\ee
Collecting the above observations and denoting
\be
A=\sum_{j=0}^{r+1}\al_j,
\ee
we have by (\ref{wf}) (where we single out a multiplicative
constant for convenience)
\be
f(z)=w(x)|\sin\th|=C \wt f(z),\qquad
C=2^{-2A-1}\exp\left(2i\sum_{j=1}^r\bt_j\arcsin\lb_j\right),
\ee
where
\be
\wt f(e^{i\theta})=e^{V(e^{i\th})}
|z-1|^{4\al_0+1}|z+1|^{4\al_{r+1}+1}
\prod_{j=1}^r  |z-z_j|^{2\al_j}|z-{z'}_j|^{2\al_j}
g_{z_j,-\bt_j}(z)z_j^{\bt_j} g_{z'_j,\bt_j}(z){z'}_j^{-\bt_j}.
\ee
Here $V(e^{i\th})=U(\cos\th)$. Thus $\wt f(z)$ is the symbol of type (\ref{fFH})
with $\Re\bt_j\in(-1/2,1/2]$. Therefore, if  $\Re\bt_j\in(-1/2,1/2)$,
$j=1,\dots,r$, we can apply
Theorem \ref{asTop} to $D_{2n}(\wt f(z))$, and obtain
\begin{multline}\la{asTop2}
D_{2n}(f(z))=C^{2n}D_{2n}(\wt f(z))=
C^{2n} \exp\left(2nV_0+\sum_{k=1}^\infty k V_k^2\right)
b_+(1)^{-4\al_0-1}b_+(-1)^{-4\al_{r+1}-1}\\
\times
\prod_{j=1}^r
b_+(z_j)^{-2(\al_j+\bt_j)}b_-(z_j)^{-2(\al_j-\bt_j)}
\times
(2n)^{2\sum_{j=1}^r(\al_j^2-\bt_j^2)
+(2\al_0+1/2)^2+(2\al_{r+1}+1/2)^2}
\mathcal{P}(z)\\
\times
\prod_{j=1}^r\frac{G(1+\al_j+\bt_j)^2 G(1+\al_j-\bt_j)^2}{G(1+2\al_j)^2}\times
\frac{G(1+2\al_0+1/2)^2}{G(1+4\al_0+1)}
\frac{G(1+2\al_{r+1}+1/2)^2}{G(1+4\al_{r+1}+1)}
\left(1+o(1)\right),\\
\mathcal{P}(z)=
\prod_{0\le j<k\le m}
|z_j-z_k|^{2(\wt\bt_j\wt\bt_k-\wt\al_j\wt\al_k)}\left({z_k\over z_j e^{i\pi}}
\right)^{\wt\al_j\wt\bt_k-\wt\al_k\wt\bt_j},\\
\Re\al_j>-{1\over 2},\qquad \Re\bt_j\in\left(-{1\over 2},{1\over2}\right),
\qquad j=0,1,\dots,r+1,
\end{multline}
where we used the fact that by the symmetry of $V(z)$, $V_k=V_{-k}$, and hence
$b_+(z_j)=b_-(z'_j)$, $b_-(z_j)=b_+(z'_j)$, $b_+(\pm1)=b_-(\pm1)$.
In the above expression for $\mathcal{P}$,
$m=2r+1$, and the points are numbered as in Theorem \ref{asTop}, namely,
$\wt z_0=z_0=1$, $\wt\al_0=2\al_0+1/2$, $\wt\bt_0=\bt_0=0$;
$\wt z_j=z_j$, $\wt\al_j=\al_j$, $\wt\bt_j=-\bt_j$, $j=1,\dots,r$;
$\wt z_{r+1}=z_{r+1}=-1$, $\wt\al_{r+1}=2\al_{r+1}+1/2$, $\wt\bt_{r+1}=\bt_{r+1}=0$;
$\wt z_j=e^{2\pi i}z_{m+1-j}^{-1}$, $\wt\al_j=\al_{m+1-j}$, $\wt\bt_j=\bt_{m+1-j}$,
$j=r+2,\dots,m$.
Expression for $\mathcal{P}$ can be written in terms of $\lb_j$. Namely,
it is not difficult to obtain by induction that
\begin{multline}
\mathcal{P}=2^{-2(\al_0+\al_{r+1}+1/4)}
\prod_{0\le j<k\le r+1}
\left|2\sin{\th_j-\th_k\over 2}\right|^{-4(\al_j\al_k-\bt_j\bt_k)}
\left|2\sin{\th_j+\th_k\over 2}\right|^{-4(\al_j\al_k+\bt_j\bt_k)}\\
\times
\prod_{j=1}^r |2\sin\th_j|^{-2(\al_j^2+\bt_j^2+\al_j)}\times
e^{2i(2A+1)\sum_{j=1}^r \bt_j\arcsin\lb_j}
e^{2\pi i\sum_{0\le j<k\le r+1}(\al_j\bt_k-\al_k\bt_j)}.
\end{multline}

Assume first that $V(z)$ is analytic.
To use (\ref{HTdet}), we need to calculate the asymptotics of
the product $\Phi_{2n}(1)\Phi_{2n}(-1)$.
In order to do this, consider $Y^{(n)}(z)$ as $z\to z_j$ in such a way that $z\in z(I)$,
where $z(I)$ is the pre-image in the $z$-plane of the sector $I$ of the $\ze$-plane
(see Figure \ref{fig2} and Section \ref{RHparam}).
Tracing back the transformations of the RHP, we obtain
\be\la{Ymid}
Y^{(n)}(z)=T(z)=S(z)
\bm
1&0\cr
f(z)^{-1}z^n& 1
\em
=
(I+R_1^{(r)}(z))P_{z_j}(z)
\bm
1&0\cr
f(z)^{-1}z^n& 1
\em,\qquad z\in z(I),
\ee
where the parametrix $P_{z_j}(z)$ at $z_j$ is (see Section \ref{RHparam}):
\be
P_{z_j}(z)=E(z)\Psi_j(\ze)F_j(z)^{-\si_3}z^{n\si_3/2},
\ee
with $E(z)$ given by (\ref{E1}).
Substituting all the expressions into (\ref{Ymid}), we obtain
\begin{multline}\la{Ymid2}
Y^{(n)}(z)=(I+R_1^{(r)}(z))\mathcal{D}(z)^{\si_3}
\bm
0&1\cr
-1&0
\em
\left(\ze^{\bt_j}F_j(z)z_j^{-n/2}\right)^{\si_3}\\
\times
\bm
e^{-i\pi(2\bt_j+\al_j)} &0\cr 0 & e^{i\pi(\bt_j+2\al_j)}
\em
\Psi_j(\ze)
\bm
F_j(z)^{-1} & 0\cr
F_j(z) f(z)^{-1} & F_j(z)
\em
z^{n\si_3/2}.
\end{multline}

Note that the expansion of $F_j^{-1}(z)$ as $z\to z_j$ is given by (\ref{etab}).
Using it we further obtain
\be
F_j(z)f(z)^{-1}=
\eta_j^{-1} e^{i\pi(\bt_j-\al_j/2)}z_j^{\al_j}u^{-\al_j}
(1+O(u)),\qquad u=z-z_j,\qquad \ze\in I.
\ee
Thus,
\be
\bm
F_j(z)^{-1} & 0\cr
F_j(z) f(z)^{-1} & F_j(z)
\em
=
\bm
e^{i\pi\al_j}& 0\cr
e^{i\pi(\bt_j-\al_j)}& e^{-i\pi\al_j}
\em
(e^{i\pi\al_j/2}z_j^{\al_j} u^{-\al_j})^{\si_3}\eta_j^{-\si_3}(1+O(u)).
\ee
To estimate $\Psi(\ze)$ for $\ze\to 0$ (i.e., $z\to z_j$), assume first that
all $\al_j\neq 0$. Substituting (\ref{psi01}) into (\ref{PsiConfl1}),
dropping the terms of order $u^{2\al_j}$ in the second column
(we will denote thus modified $Y(z)$ by $\wt Y(z)$)
we obtain the following limit for the combination needed in (\ref{Ymid2})
(here tilde over the limit sign means that we have to
drop $u^{2\al_j}$ terms before taking the limit):
\begin{multline}
\wt\lim_{u\to 0}
\bm
e^{-i\pi(2\bt_j+\al_j)} &0\cr 0 & e^{i\pi(\bt_j+2\al_j)}
\em
\Psi_j(\ze)
\bm
e^{i\pi\al_j}& 0\cr
e^{i\pi(\bt_j-\al_j)}& e^{-i\pi\al_j}
\em
(e^{i\pi\al_j/2}z_j^{\al_j} u^{-\al_j})^{\si_3}\\
=
e^{i\pi(\al_j/2-\bt_j)\si_3} M n^{\al_j\si_3},
\end{multline}
where
\be
M=
\bm
\left[ e^{i\pi\al_j}{1\over\Ga(\bt_j-\al_j)}-
 e^{-i\pi\al_j}{\Ga(1+\al_j-\bt_j)\over \Ga(\al_j+\bt_j)\Ga(1-\al_j-\bt_j)}
\right]\Ga(-2\al_j)e^{i\pi\bt_j}&
-{\Ga(2\al_j)\over \Ga(\al_j+\bt_j)}\cr
\left[ e^{-i\pi\al_j}{1\over\Ga(-\bt_j-\al_j)}-
 e^{i\pi\al_j}{\Ga(1+\al_j+\bt_j)\over \Ga(\al_j-\bt_j)\Ga(1-\al_j+\bt_j)}
\right]\Ga(-2\al_j)e^{i\pi\bt_j}&
{\Ga(2\al_j)\over \Ga(\al_j-\bt_j)}
\em.
\ee
This expression can be simplified. Namely, the $11$ matrix element
\begin{multline}
M_{11}=
e^{i\pi(\bt_j-\al_j)}{\Ga(-2\al_j)\over \Ga(\bt_j-\al_j)}
\left(e^{2\pi i\al_j}-{\sin\pi(\bt_j+\al_j)\over \sin\pi(\bt_j-\al_j)}
\right)\\
={\Ga(-2\al_j)\over \Ga(\bt_j-\al_j)}{\sin(-2\pi\al_j)\over \sin\pi(\bt_j-\al_j)}
={\Ga(1+\al_j-\bt_j)\over \Ga(1+2\al_j)}.
\end{multline}
Similarly,
\be
M_{21}= {\Ga(1+\al_j+\bt_j)\over \Ga(1+2\al_j)}.
\ee
Thus
\be\la{M}
M=
\bm
{\Ga(1+\al_j-\bt_j)\over \Ga(1+2\al_j)}&
 -{\Ga(2\al_j)\over \Ga(\al_j+\bt_j)}\cr
{\Ga(1+\al_j+\bt_j)\over \Ga(1+2\al_j)}&
 {\Ga(2\al_j)\over \Ga(\al_j-\bt_j)}
\em.
\ee
Substituting the just found limit and (\ref{mub}) for
$(\mathcal{D}(z)/(\ze^{\bt_j}F_j(z))^2$
into (\ref{Ymid2}), we obtain
\be\la{Yend}
\wt Y^{(n)}(z_j)=(I+r^{(n)}_j)L^{(n)}_j,\qquad
L^{(n)}_j=\bm
M_{21}\mu_j\eta_j^{-1} n^{\al_j-\bt_j} z_j^{n} &
M_{22}\mu_j\eta_j n^{-\al_j-\bt_j}\cr
-M_{11}\mu^{-1}_j \eta_j^{-1} n^{\al_j+\bt_j} &
-M_{12}\mu^{-1}_j \eta_j n^{-\al_j+\bt_j} z_j^{-n}
\em,
\ee
where $r_j=R_1^{(r)}(z_j)$, and $\eta_j$, $\mu_j$ are given
by (\ref{eta},\ref{mu}).

Note that the matrix $L^{(n)}_j$ has the structure
\be\la{LhatL}
L^{(n)}_j=n^{-\bt_j\si_3}\widehat L^{(n)}_j n^{\al_j\si_3},
\ee
where $\widehat L$ depends on $n$ only via the oscillatory terms $z_j^n$.

From (\ref{RHM}) and (\ref{Yend}) at $z_j=1$,
\be
\Phi_{2n}(1)=Y^{(2n)}_{11}(1)=L^{(2n)}_{0,11}(1+O(n^{-2\max_k\bt_k-1})).
\ee
From (\ref{Yend},\ref{M},\ref{mu},\ref{eta}), 
using the doubling formula for the $\Ga$-function
\be\la{Gad}
\frac{\Ga(1+x)}{\Ga(1+2x)}=\frac{\sqrt{\pi}}{2^{2x}\Ga(x+1/2)},
\ee
we obtain the following main term of $\Phi_{2n}(1)$:
\begin{multline}
L^{(2n)}_{0,11}=M_{0,21}\mu_0\eta_0^{-1}(2n)^{2\al_0+1/2}=\\
{\sqrt{\pi}e^{(V_0-V(1))/2+i\sum_{j=1}^r(\pi-\th_j)\bt_j}
\over 2^{4\al_0+1}\Ga(1+2\al_0)}
\prod_{j=1}^r \left|2\sin{\th_j\over 2}\right|^{-2\al_j}
2^{2(\al_0-\al_{r+1})} n^{2\al_0+1/2}.
\end{multline}
Similarly, we obtain
\begin{multline}
\Phi_{2n}(-1)=L^{(2n)}_{r+1,11}(1+O(n^{-2\max_k\bt_k-1})),\\
L^{(2n)}_{r+1,11}=
{\sqrt{\pi}e^{(V_0-V(-1))/2-i\sum_{j=1}^r\th_j\bt_j}
\over 2^{4\al_{r+1}+1}\Ga(1+2\al_{r+1})}
\prod_{j=1}^r \left|2\cos{\th_j\over 2}\right|^{-2\al_j}
2^{-2(\al_0-\al_{r+1})} n^{2\al_{r+1}+1/2}.
\end{multline}
Therefore
\begin{multline}\la{PP}
\Phi_{2n}(1)\Phi_{2n}(-1)=
{\pi e^{V_0-(V(1)+V(-1))/2+2i\sum_{j=1}^r\bt_j\arcsin\lb_j}
\over 2^{4(\al_0+\al_{r+1})+2}\Ga(1+2\al_0)\Ga(1+2\al_{r+1})}\\
\times
\prod_{j=1}^r
|2\sin\th_j|^{-2\al_j} n^{2(\al_0+\al_{r+1})+1}(1+O(n^{-2\max_k\bt_k-1})).
\end{multline}

Substituting (\ref{asTop2}) and (\ref{PP}) into (\ref{HTdet}) we obtain
(\ref{asDH}) squared. We use the following
observations in the process:

\begin{itemize}
\item
Since $V_k=V_{-k}$,
\[
b_+(\pm1)=e^{(V(\pm1)-V_0)/2}.
\]

\item
The following elementary identity holds
\begin{multline}
\prod_{0\le j<k\le r+1}
\left|2\sin{\th_j-\th_k\over 2}\right|^{-(\al_j\al_k-\bt_j\bt_k)}
\left|2\sin{\th_j+\th_k\over 2}\right|^{-(\al_j\al_k+\bt_j\bt_k)}\\
=
2^{-\sum_{0\le j<k\le r+1}\al_j\al_k} \prod_{0\le j<k\le r+1}
|\lb_j-\lb_k|^{-(\al_j\al_k+\bt_j\bt_k)}
\left|\lb_j\lb_k-1+\sqrt{(1-\lb_j^2)(1-\lb_k^2)}\right|^{\bt_j\bt_k}
\end{multline}

\item Applying the doubling formula (\ref{Gdouble}) we easily obtain that
\be
\frac{G(1+2\al+1/2)^2}{G(1+4\al+1)}\Ga(1+2\al)=
2^{-8\al^2-2\al}\pi^{2\al+1}{G(1/2)^2\over G(1+2\al)^2}.
\ee

\end{itemize}

If $V(z)\equiv V_r(z)$ is real-valued for $z\in C$, and $\al_j\in\bbr$,
$i\bt_j\in\bbr$, $j=0,\dots,m$, then the weight $f(z)$, $z\in C$, is positive,
and therefore $D_n(w)$ is positive. Then (\ref{asDH}) represents the correct branch
of the square root. Since $D_n(w)$ is continuous in $\al_j$, $\bt_j$, and the parameter
$h$ in $V(z)=V_r(z)+(V(z)-V_r(z))h$, $h\in [0,1]$, and the error term is uniform
in these parameters (see Section \ref{RRHP}), the formula (\ref{asDH}) has the correct
sign in general. This finishes the proof for analytic $V(z)$. The extension to
smooth $V(z)$ is carried out similarly to the argument in the previous section
by using the standard multiple-integral representation of a Hankel determinant.

\section*{Acknowledgements}
P. Deift was supported in part by NSF grants \# DMS 0500923 and \#
DMS 1001886. A. Its was supported in part by NSF grant \#DMS 0701768
and EPSRC grant \#EP/F014198/1. I. Krasovsky was supported in part
by EPSRC grants \#EP/E022928/1 and \#EP/F014198/1.

%%%%%%%%%%%%%%%%%%%%%%%%%%%%%%%%%%%%%%%%%%%%%%%
%%%%%%%%%%%%%%%%%%%%%%%%%%%%%%%%%%%%%%%%%%%%%%%%
%%%%%%%%%%%%%%%%%%%%%%%%%%%%%%%%%%%%%%%%%%%%%%%
%%%%%%%%%%%%%%%%%%%%%%%%%%%%%%%%%%%%%%%%%%%%%%

%%%%%%%%%%%%%%%%%%%%%%%%%%%%%%%%%%%%%%%
%%%%%%%%%%%%%%%%%%%%%%%%%%%%%%%%%%%%%%%
%%%%%%%%%%%%%%%%%%%%%%%%%%%%%%%%%%%%%%%
%%%%%%%%%%%%%%%%%%%%%%%%%%%%%%%%%%%%%%%
%%%%%%%%%%%%%%%%%%%%%%%%%%%%%%%%%%%%%%%
%%%%%%%%%%%%%%%%%%%%%%%%%%%%%%%%%%%%%%%
%%%%%%%%%%%%%%%%%%%%%%%%%%%%%%%%%%%%%%%
%%%%%%%%%%%%%%%%%%%%%%%%%%%%%%%%%%%%%%%
%%%%%%%%%%%%%%%%%%%%%%%%%%%%%%%%%%%%%%%
%%%%%%%%%%%%%%%%%%%%%%%%%%%%%%%%%%%%%%%
%%%%%%%%%%%%%%%%%%%%%%%%%%%%%%%%%%%%%%%

\end{document}